\documentclass[11pt,reqno]{amsart}
\usepackage{amsmath}
\usepackage{amssymb}
\usepackage{amstext}
\usepackage{amsfonts}
\usepackage{amssymb}

\usepackage[usenames, dvipsnames]{color}
\theoremstyle{plain}
\newtheorem{theorem}{Theorem}[section]
\newtheorem{lemma}[theorem]{Lemma}
\newtheorem{corollary}[theorem]{Corollary}
\newtheorem{proposition}[theorem]{Proposition}
\newtheorem{example}[theorem]{Example}
\newtheorem{definition}[theorem]{Definition}
\newtheorem{remark}[theorem]{Remark}
\theoremstyle{plain}
\newtheorem{prop}[theorem]{Proposition}

\newcommand{\supp}{\mathrm{supp}}

\def \R {{\mathbb {R}}}

\newcommand{\Gi}{\mathbb{G}}

\newcommand{\average}{{\mathchoice {\kern1ex\vcenter{\hrule height.4pt
width 6pt
depth0pt} \kern-9.7pt} {\kern1ex\vcenter{\hrule height.4pt width 4.3pt
depth0pt}
\kern-7pt} {} {} }}

\newcommand{\N}{\mathcal{N}}
\def\p{\partial}

\newcommand{\Y}{Y}

\def \R {{\mathbb {R}}}
\def \N {{\mathbb N}}

\def  \e {{\varepsilon}}

\def \tilde {\widetilde}

\def \xn {{x}}
\def \yn {{y}}

\def \Gn {{\Gamma}}

\def \x {{\hat x}}
\def \y {{\hat y}}
\def \z {{\hat z}}

\def \R {{\mathbb {R}}}

\def \N {{\mathbb {N}}}

\def \e {{\varepsilon}}

\def \phi {{\varphi}}

\def \tilde {\widetilde}

\numberwithin{equation}{section} \makeatletter

\title[Schauder estimates]{Schauder estimates at the boundary \\ for sub-Laplacians in Carnot groups }
\author{Annalisa Baldi, Giovanna Citti, Giovanni Cupini}
\address{Dipartimento di Matematica, Piazza di Porta S. Donato 5,
40126 Bologna, Italy}\email{annalisa.baldi2@unibo.it, giovanna.citti@unibo.it, giovanni.cupini@unibo.it}

\thanks{ 
The authors are supported by University of Bologna, funds for selected research 
topics, and by the People
Program (Marie Curie Actions) of the European Union's Seventh
Framework Program FP7/2007-2013/ under REA grant agreement n. 607643. 
The first and the third author are supported by GNAMPA of INdAM, Italy. 
}

\begin{document}

\begin{abstract} 
In this paper we prove Schauder estimates at the boundary for sub-Laplacian type operators in Carnot groups. While internal Schauder estimates have been deeply studied, up to now subriemannian estimates at the boundary are 
known only in the Heisenberg groups. The proof of these estimates in the Heisenberg setting, 
due to Jerison (\cite{Jerison}),  
is based on the Fourier transform technique and can not be repeated in general Lie groups. After the result of Jerison no new contribution to the boundary problem has been provided. 
In this paper we introduce a new approach, which allows to build a Poisson kernel starting from the fundamental solution, from which we deduce the Schauder estimates at non 
characteristic boundary points. 
\end{abstract}

\maketitle

\tableofcontents

 \section{Introduction}

\subsection{Aim of this work}
The aim of this work is to prove Schauder estimates at the boundary for sub-Laplacian type operators in Carnot groups. 

As it is well known, Schauder estimates at the boundary in the Euclidean setting are based on two main ingredients. The first one, which is the core of the Schauder method, is the local reduction of general uniformly elliptic operators to the Laplace operator. The second one, which seems elementary in the Euclidean setting, is a  reflection technique which reduces the boundary Schauder estimates to internal ones. 
Unfortunately, this technique can not be applied in the strong anisotropic setting of a Carnot group, since a Laplace type operator 
in this framework is not invariant with respect to reflection, nor can be approximated by any invariant operator. 
In the special case of the Heisenberg group, Schauder estimates are a classical result due to Jerison (see \cite{Jerison}), 
but not even this technique, based on the standard Fourier transform, can be extended to general Lie groups, whose geometry is not related to the Fourier transform. 
After that, no new contribution has been provided to the problem, which is still open, 
while its solution would be necessary for the development 
of nonlinear PDE's theory in this setting.  

In this paper we introduce a completely different approach, which is new even in the Riemannian setting, and which allows to built a Poisson kernel starting from the knowledge of a smooth fundamental solution for the problem on the whole space. 
\subsection{Carnot groups}\label{introcarnot}

A Carnot group $\Gi$ can be identified with $\R^n$ with a polynomial group law $(\Gi,\cdot)$, 
whose Lie algebra ${\mathfrak{g}}$  admits a  step $\kappa$ stratification. Precisely 
there exist linear subspaces $V^1,...,V^\kappa$ such that
\begin{equation}\label{strat}
\mathfrak{g}=V^1\oplus...\oplus V^\kappa,\quad [V^1,V^{i-1}]=V^{i},{\textrm{ if }} i\leq\kappa, 
\quad [V^1,V^{\kappa}]=\{0\}.
\end{equation}
We will call horizontal tangent bundle the subspace $V^1$, and we will choose a basis for it, denoted 
$\left\{X_1, \cdots, X_m\right\}$. By the assumption on the Lie algebra, 
this basis can be completed to a basis $\left\{X_1, \cdots, X_n\right\}$ of $\mathfrak{g}$
with the list of their commutators. On the vector space $V^1$ we define a Riemannian metric 
which makes orthonormal the vector fields $X_1, \cdots, X_m$. 
Several equivalent  left invariant distances 
$d$ can be introduced on the whole space 
requiring that their restriction to $V^1$ is equivalent to the 
fixed Riemannian metric
 (see for example Nagel, Stein and Wainger in \cite{NSW}).
The subriemannian gradient of a regular function $f$ will be denoted $\nabla f=(X_1f, \cdots, X_m f)$ and $f$ will be  called of class $C^1$ if this gradient is continuous 
with respect to the distance $d$. 
More generally, spaces of H\"older continuous functions $C^{k, \alpha}$ can be defined
in terms of this distance and this gradient. We will study here a subelliptic operator 
expressed as follows:
\begin{equation}\label{laplacoperator}\Delta = \sum_{i=1}^m (X_i^2 + b_i X_i),\end{equation}
with regular coefficients $b_i$. 
Operators of this type are hypoelliptic and have been deeply studied after the first works of  Folland
and  Stein \cite{FollandStein}, Rothschild and Stein \cite{RS}, Jerison and Sanchez-Calle \cite{JSC}, 
 Fefferman and Sanchez-Calle \cite{FSC}, 
Kohn and Nirenberg \cite{KN},  and 
 Jerison \cite{Jerison, Jerison2} (see also \cite{BLU} for a recent monograph).  
 Their fundamental solution 
$\Gamma_\Delta$ is of class $C^\infty$ far from the diagonal 
and it can be estimated in term of the distance 
as follows \begin{equation}\label{gammabehavior}\Gamma_\Delta(x,y)\approx \frac{1}{d^{Q-2}(x,y)},\end{equation}
for a suitable integer $Q$, called homogeneous dimension of the space (see \eqref{homodim} for a precise definition). A kernel with the behavior of $\Gamma_\Delta$ is called of local type 2. 
In general we will say that a kernel $K$
is of local type $\lambda$ with respect to the distance $d$ if 
for every open bounded set $V$ and
for every $p \geq 0$ there exists a positive constant $C_p$ such that, for every $x, y \in V$, with $x\not= y$ 
\begin{equation}
	\label{e:sileva}
|X_{i_1}\cdots X_{i_p}K(x, y)|\leq C_p  d(x,  y)^{\lambda-p-2}  \Gamma_\Delta(x,y).\end{equation}
A well established theory of singular integrals in H\"ormander setting (due to Folland and Stein 
\cite{FollandStein}, Rothschild and Stein \cite{RS}, Greiner and Stein \cite{GreinerStein}) allows to prove interior Schauder estimates. 
For more recent results we quote the 
H\"older estimates by Citti \cite{C}, the Schauder estimates of  Xu \cite{Xu} and Capogna
and Han \cite{CapognaHan} for uniformly subelliptic operators,  
Bramanti and Brandolini \cite{BramantiBrandolini}
for heat-type operators and the results of Lunardi \cite{Lunardi}, Di
Francesco and Polidoro \cite{PolidoroDiFrancesco}, Gutierrez and Lanconelli \cite{GutierrezLanconelli},  Bramanti and Zhu \cite{bramantizhu} and Simon \cite{Simon}
for  a large 
class of operators. The problem at the boundary is completely different and largely unsolved. 

\subsection{Schauder estimates at the boundary}

A surface $M$ in a Carnot group, smooth in the Euclidean sense,
 can be locally expressed as the zero 
level set of a function $f\in C^\infty$, but  it can 
have points in $M$ where its  subriemannian gradient 
vanishes. At these points, called characteristic, 
the geometry of the surface is not completely understood. 
Far from characteristic points,  
properties of regular surfaces have been largely studied starting 
from the papers of Kohn and Nirenberg 
in \cite{KN}, Jerison in \cite{Jerison} and more 
recently by Franchi, Serapioni and Serra Cassano, \cite{FSSC, FSSC1} 
(see also the references therein). 
The stratification defined in (\ref{strat})
induces a stratification on the tangent plane of the manifold $M$. We will call ${\hat V}^1 = V^1 \cap TM$, 
${\hat V}^2 = V^2 \cap TM$, $\cdots$,  ${\hat V}^\kappa = V^\kappa \cap TM$. It is not restrictive to assume 
that $X_1\in V^1$ is normal to ${\hat V}^1$ with respect to the metric fixed in $V^1$ so that we can denote by
$\{\hat X_i\}_{i=2,\cdots, m}$ a basis of ${\hat V}^1$. 
We also require that the following condition holds:
\begin{equation}\label{assumption}Lie ({\hat V}^1) = TM.\end{equation}
Under this assumption the manifold $M$ has a H\"ormander structure, 
and $\hat V^1$ inherits a metric from the immersion in $V^1$. Hence  
a distance $\hat d$ and corresponding classes of H\"older continuous functions ${\hat C}^{k, \alpha}(M)$ are well defined. For every choice of regular coefficients $(b_i)_{i=2,\cdots, m}$, 
a Laplace-type operator 
\begin{equation}\label{laplacefundamental}\hat \Delta = 
\sum_{i=2}^m \hat X_{i}^2 + \sum_{i=2}^m b_i \hat X_i
\end{equation} 
is  defined on $M$, 
with fundamental solution 
$\hat \Gamma_{\hat\Delta}.$

It has been proved by Kohn and Nirenberg in \cite{KN} that, 
if $D$ is a smooth open set with smooth boundary and $g$ a smooth function defined on the boundary of $D,$ the problem
\begin{equation}\label{Poisson_problem}\Delta u=0  \ \text{in } D, \quad u=g \ \text{on }\partial D\end{equation}
has a unique solution, of class $C^{\infty}$ up to the boundary at non characteristic points. At the characteristic points very few results are known (see \cite{Jerison2}, already quoted, and \cite{CG98}, \cite{GV2000} and \cite{V}, where existence of non tangential limits up to the boundary are established).
In this paper we prove the exact analogous of the classical Schauder estimates 
at the boundary, providing estimates of the ${\hat C}^{2, \alpha}$ norm of the 
solution in terms of the H\"older norm of the data.  
Precisely our result can be stated as follows.

\begin{theorem} \label{c:schauderGroups} 
Let $D \subset \mathbb{G}$ be a smooth, bounded domain and 
assume that the vector fields 
$\{X_i\}_{i=1, \cdots, m}$ satisfy the assumption \eqref{assumption}. 
Denote $u$ the unique
solution to
$$\Delta u=f\; \text{in}\  D, \quad u= g \text{ on }\,  \partial D, $$
where $f \in C^\alpha(\bar D)$ and $g \in \hat C^{2, \alpha} (\partial D)$ and $\alpha>0$. 
If $\bar x\in \partial D$ and $V$ is an open neighborhood of $\bar x$ without characteristic
 points, for every $\phi\in C^\infty_0(V)$ we have 
$\phi u \in C^{2, \alpha}(\bar D\cap V)$ and 
\begin{equation}
	\label{stime}
\|\phi u\|_{C^{2, \alpha}(\bar D\cap V)} \leq C (\|g\|_{\hat C^{2, \alpha} (\partial D)} + \|f\|_{C^\alpha(\bar D)}).\end{equation}
\end{theorem}

As we mentioned before, up to now subriemannian boundary
Schauder estimates are known only for the Heisenberg group (see \cite{Jerison}) and are based on the construction of a Poisson kernel. 
If $D$ is an open bounded set, and 
 $V$ is a neighborhood of a non characteristic point $\overline x\in\partial D$, 
we say that $P:C^{\infty}(\partial D\cap V)\rightarrow C^{\infty}(V\cap \overline{D})$ is a local Poisson operator 
for the   problem \eqref{Poisson_problem}
 if,  for every $g\in C^{\infty}(\partial D\cap V)$, the function $u=P(g)$
satisfies 
$\Delta u=0$ in $D\cap V$ and  $u(x)=g(x)$ for all $x\in \partial D\cap V$. 

The construction of the Poisson kernel contained  in \cite{Jerison} is based on 
the standard Fourier transform and can not be 
 directly repeated in general Lie groups. 
General measure theory ensures the existence of a Poisson kernel under very weak assumptions on the vector fields  
(see for example Lanconelli and Uguzzoni  \cite{LAU}), but  this 
theory only allows to establish $L^p$ regularity of the solution at the boundary. 
A Poisson kernel has been built by 
 Ferrari and Franchi  \cite{Ferrarifranchi} in the special case  when the set $D$ is a half space of the form 
 $ \R^+\times\hat \Gi $. In this case, as well as in the Heisenberg groups considered in 
\cite{Jerison}, assumption \eqref{assumption} is satisfied. 
However, this assumption is verified by a 
much larger class of Carnot groups, where the splitting  of $\Gi$ is not possible. A simple example can be the following one: 
\begin{equation}\label{campesempio}X_1 = \partial_1, \quad X_2 = \partial_2 + x_1^2 \partial_5 + x_3 \partial_4, 
\quad X_3 = \partial_3 + x_4 \partial_5, 
\end{equation}
with $M=\{x_1=0\}$. 

Our construction of the 
Poisson operator is based on the knowledge of a smooth fundamental solution, its restriction to the boundary, and on the properties of singular integrals. Since our result is local, we can locally express the 
boundary of $D$ as the graph of a smooth function $w$, and 
perform a change of variable, to reduce the boundary to a plane. 
In the new coordinates the vector fields will explicitly depend on the 
function $w$ defining the boundary, and will not be homogeneous in general.
For sub-Laplacian type operators associated to these vector fields we will obtain the following expression of the 
Poisson kernel. 
\begin{theorem}
 \label{mainRn}
Let $D= \{(x_1,\x)\in \R \times \R^{n-1}: x_1>0\}\subset \Gi$ 
be  a non
characteristic half space and let $g\in C^\infty(\partial D)$. Let $\bar x\in \partial D$, 
let $V_0$ be a neighborhood of $\bar x$  in $\R^{n}$ and let 
\begin{equation}K_1(g)(\hat y) :=\int_{\partial D \cap V_0} \Gamma_{\Delta}((0, \hat y), (0,\hat z)) \hat \Delta g(\hat z) 
d\hat z.\end{equation} There exists  
a lower order operator $R$ of type  3/2 with respect to the distance $\hat d$ defined on $\partial D$, such that 
for every  neighborhood $V$ of $\bar x$  in $\R^{n}$, $V\subset\subset V_0$, 
 the operator 
\begin{equation}\label{poissonintro}P(g)(x) := \int_{\partial D \cap V_0} \Gamma_{\Delta}(x, (0,\hat y)) (K_1 + R)(g)(\hat y) d\hat y \end{equation}
 is a Poisson kernel in $V$. 
\end{theorem}

The representation \eqref{poissonintro}
and the properties of the fundamental solution 
immediately ensure that $P(g)$ satisfies the equation in \eqref{Poisson_problem}.
In order to show that $P$ is a Poisson operator, we only have to 
show that  $P(g) = g$  on the boundary $\{x_1=0\}$. 
Denoting by
$E_{\Gamma_{\Delta}(0, \cdot)}$ the operator associated to the 
kernel $\Gamma_{\Delta}((0,\x), (0,\z))$, 
this is equivalent to say that 
$K_1 + R$ is the inverse of the operator $E_{\Gamma_{\Delta}(0, \cdot)}$.
Under the assumption \eqref{assumption} this is proved using 
the fundamental solution $\hat \Gamma_{\hat\Delta}$ of the operator $\hat \Delta$ defined in
\eqref{laplacefundamental}. Indeed  $\hat \Gamma_{\hat\Delta}$ 
satisfies the following approximate reproducing formula:
\begin{theorem}\label{teorema1} Let $D= \{(x_1,\x)\in \R \times \R^{n-1}: x_1>0\}\subset \Gi$ be a non
characteristic half plane. 
If $\bar x\in \partial D $, then there exists a neighborhood $V$ of $\bar x$ in $\Gi$  such that
the fundamental solution admits the following representation: 
\begin{equation}\label{tesi1}\hat {\Gamma}_{\hat\Delta}(\x, \y)  =\int_{\partial D \cap V}  
\Gamma_{\Delta}((0,\x), (0,\z))
\Gamma_{\Delta}((0,\z), (0,\y)) d\z + \hat R_{\hat\Delta}(\x, \y), \end{equation}
for every $x = (0, \x), y= (0, \y)\in \partial D \cap V$, 
where $\hat R_{\hat\Delta}$ is a kernel of type  $5/2$ with respect to the distance $\hat d$.
\end{theorem}

This theorem ensures
 that  $K_1$ is the inverse of  the operator $E_{\Gamma_{\Delta}(0, \cdot)}$ up 
to a remainder. The proof of Theorem \ref{mainRn} will be concluded 
with a standard version of the parametrix method, which allows to 
carefully handle the remainder and to prove that $K_1 +R$ is indeed the inverse of 
$E_{\Gamma_{\Delta}(0, \cdot)}$. 

Theorem \ref{teorema1} expresses 
$E_{\Gamma_{\Delta}(0, \cdot)}$ as 
the square root of 
the operator associated to 
$\hat {\Gamma}_{\hat\Delta}$. This result, well known in the 
Euclidean setting (see for example the results of 
Caffarelli and Silvestre \cite{CS}), was not known 
for general Carnot groups, but only in the special case when 
the group $\Gi$ is expressed as $\Gi= \R \times \hat \Gi$ (see Ferrari and Franchi in \cite{Ferrarifranchi}). 
The proof in this setting is inspired by the results of Evans in \cite{E} (in the Euclidean setting) and 
 of Capogna, Citti and Senni (in Carnot groups) in \cite{CCS}.

\subsection{Structure of the paper and sketch of the proofs} 
The paper starts with Section 2, where we fix notations and recall known properties of Carnot groups and their Riemannian approximation. 

In Section 3  we show that a 
non characteristic plane can always be represented as the plane  $\{(x_1,\x)\in \mathbb{R}\times \mathbb{R}^{n-1}\,:\, x_1=0\}$  
with the canonical exponential change of variables described in 
\eqref{deftheta}. In these coordinates the vector fields attain an explicit 
polynomial representation recalled in 
\eqref{struttura campi}.  Moreover, Section 3  contains the proof of 
Theorem \ref{teorema1} under the assumption that the boundary of $D$ 
is a non characteristic plane and the vector fields are homogeneous. 
The proof of this theorem  is the most technical part of the paper and it is based on 
a Riemannian approximation and a parabolic regularization of the operator $\Delta$. 
Precisely the Riemannian approximation of the Laplace type operator $\Delta$ is an operator of the form
\begin{equation}\label{Definede}\Delta_{\e}= 
\Delta+ \e^2 \sum_{i=m+1}^{n}X_i^2,\end{equation}
and its parabolic regularization leads to the operator
\begin{equation}\label{e:L_epsilon}
 L_\e:= \partial_t - \Delta_{\e}.\end{equation}
In a neighborhood of any non characteristic point $z$ of the plane $\partial D$ 
we will apply a new version of the freezing and parametrix methods to approximate  
the fundamental solution $\Gamma_{\e}$ of  $L_\e$ in terms of the fundamental solution $\hat \Gamma_{\e}$ of a 
suitable tangential heat operator $\partial_t - \hat \Delta_{\e}$.  
The parametrix method has already been largely used in the subriemannian setting 
for estimating the fundamental solution 
in terms of a known one (see for example  \cite{RS, SC, JSC, C, blu_a}). 
Here we are inspired by the papers \cite{CCS} and \cite{CM} where the relation between the fundamental solution on the whole space and its restriction to the boundary was studied in the framework of a diffusion driven motion by curvature. 
The main technical difficulty in our setting is due to the fact that neither the geometry of the subriemannian space nor the structure of the subriemannian operators is naturally represented as the direct sum of the tangential 
and the normal part.  This splitting is true in the Riemannian approximation, and this is the reason for using this approximation. 
However the subriemannian structure and its Riemannian approximation have different homogeneous dimension.
Hence we need to introduce a non homogeneous version of the parametrix method, which leads to the existence of a constant $C$ such that  
 \begin{equation}
	 \label{quellosopra}
\left| \Gn_\e((0, \x, t), (0, \y,\tau)) - 
 \frac{\hat \Gamma_{\e}((\x,t), (\y, \tau))}{
 \sqrt{t-\tau}} \right|\leq 
C\Gn_{\e}((0, \x, t), (0, \y, \tau))(t-\tau)^{1/4}
\end{equation}
 for every $\x$ and $\y\in \partial D$, (this is done in Proposition \ref{lemmaJe2} below). 
The key point here is to prove that $C$ is independent of $\e$. 
The proof  is quite delicate, and it is based on an interplay between the Riemannian and subriemannian 
nature of our operators.  
Since all 
constants in \eqref{quellosopra} are independent of $\e$, we can let $\e$ go to $0$ and obtain 
an analogous estimate for subriemannian operators. 
Denoting $\Gamma$ the fundamental solution of $\partial_t - \Delta$ and 
$\hat \Gamma$ the fundamental solution of the operator $\partial_t - \hat \Delta$, 
we will prove in 
Theorem \ref{lemmaJe1} that 
there  exists  a constant $C=C(T)$  such that 
 for 
all  $z=(0, \z)$, $x=(0, \x)$ in $\partial D$ and for every $t$ and $\tau$, with $0<t-\tau<T$, we have 
\begin{equation}
	 \label{altraggiunta}
\left| \Gn((0, \x, t), (0, \y,\tau)) - 
 \frac{\hat \Gamma((\x,t), (\y, \tau))}{
 \sqrt{t-\tau}} \right|\leq 
C\Gn((0, \x, t), (0, \y, \tau))(t-\tau)^{1/4}
\end{equation}

Now, integrating in the time variable, we deduce the proof of Theorem \ref{teorema1} for homogeneous 
vector fields and on a plane (also called Lemma \ref{l:Gammaconvolve}). 

In Section 4 we provide the full proof of Theorem \ref{teorema1} on smooth manifolds. Since 
this is a local result, we show that, via a suitable change of variables, it is possible to identify the boundary of $D$ with the plane $\{x_1 =0\}$.
 With this change of variables the vector fields  $(X_{i})$ become non homogeneous, but they still define an H\"ormander structure. In Section 4.1  we describe this procedure and recall some properties of subriemannian spaces in this generality. 
Then, in Section 4.2 we apply a new simplified version of the parametrix method 
of Rothschild and Stein \cite{RS} tailored on the present setting, 
and locally we reduce the vector fields to homogeneous ones. 
With this instrument we can deduce the proof of Theorem \ref{teorema1} for 
smooth surfaces from the  one on planes, previously proved in Section 3. 

Finally, Section 5 contains  the construction of the Poisson kernel, which allows to prove  Theorem \ref{mainRn}. 
The main idea of the proof of this theorem has been outlined above. First, we use 
Theorem \ref{teorema1} to build an approximated kernel. 
After that, a standard version of the parametrix method is applied to 
obtain the Poisson kernel from the approximating one. 
The Schauder estimates stated in 
Theorem \ref{c:schauderGroups} are a consequence of the boundedness of the 
operator associated to the Poisson kernel and they will be proved with the same instruments as in \cite{Jerison}. 

\section{Notations and known results}\label{section2}

\subsection{The subriemannian structure}
As recalled in Section \ref{introcarnot}, a Carnot group $\Gi$ is 
$\R^{n}$, with the group low induced by the exponential map and the stratification 
$ V^1\oplus \cdots \oplus V^\kappa$ of the tangent space 
recalled in 
\eqref{strat}. 
The stratification induces a natural notion of 
degree of a vector field:  
\begin{equation}\label{degree}deg(X)=j \quad\text {whenever}\; X\in V^j.\end{equation}
If $\{X_i\}_{i=1, \cdots, n}$ is the stratified basis introduced 
in subsection \ref{introcarnot}, we will write also $deg(i)$ instead of $deg(X_i)$. 
Via the exponential map, 
$\R^{n}$ is endowed with a Lie  group structure  
and the resulting group is denoted by $\mathbb{G}$. Since in this setting the exponential 
map around a fixed point $y$ is a global diffeomorphism, every  other point 
$x$ can be uniquely represented as 
$\xn = \exp(v_1 X_1)\exp(\sum_{i=2}^n v_{i} X_i)(y). $
Consequently we can define a logarithmic function 
$\Theta_{X_1, \cdots, X_n, y}$  as the inverse of the 
exponential map: 
\medbreak\noindent
\begin{equation}\label{deftheta}
\Theta_{X_1, \cdots, X_n, y}: \Gi \to \mathfrak
g,  \quad \Theta_{X_1, \cdots, X_n, y}(\xn)= (v_1, \cdots, v_n).\end{equation}
We will simply denote $\Theta_y$ instead of $\Theta_{X_1, \cdots, X_n, y}$ when no ambiguity may arise. 
Note that we are using exponential canonical coordinates of second type around a fixed point $y\in \mathbb{G}$, 
which will simplify notations while dealing with a boundary problem. 

In particular, the fixed point $y$, around which we choose the axes, has coordinates $0$ and the vector field $X_1$ is  represented as $X_1=\partial_1$ and
all the others vector fields $(X_i)_{i\geq 2}$ coincide with the partial derivative $\partial_i$ 
at the fixed point $y=0$. In any other point they can be represented in these coordinates as 
\begin{equation}\label{struttura campi}
X_1=\partial_1, \quad X_i=\partial_i+\sum_{deg(j)> deg(i)}a_{ij}(v)\partial_j\quad i=2,\cdots,n,\end{equation}
where $a_{ij}$ are 
homogeneous polynomials 
of degree $deg(j) - deg(i)$ depending only on variables $v_h$, with $deg(h) \le deg(j)-deg(i)$ (see for example \cite{RS} for a detailed proof).
Note that if $deg(i)=\kappa $ then $X_i=\partial_i$.

By construction the vectors $\{X_{i}\}_{i=1, \cdots, m}$ 
and their commutators span
 $\mathfrak g$ at every point, and consequently verify H\"ormander's finite rank condition
(\cite{hormander}). 
 Due to the stratification of the algebra, a natural 
family of dilation  $(\delta_\lambda)_{\lambda>0}$ acts on points 
$v= \sum_{i=1}^n v_i X_i \in \mathfrak{g} $ as follows:
\begin{equation}\label{dilat}
\delta_\lambda(v):=\lambda^{deg(i)}v_i. 
\end{equation} 
On $V^1$, which is generated by $X_1, \cdots X_m$, we define a Riemannian metric which makes $X_1, \cdots, X_m$ an orthonormal basis. 
The associated norm will be extended to an homogeneous norm to the whole $\mathfrak g$ 
defined as follows:  
\begin{equation}\label{subnorm}\|v\| := \sum_{i=1}^{n}  | v_{i}|^{1/{deg(i)}}.\end{equation}
Via the logarithmic function defined in (\ref{deftheta}) the dilation $\delta_{\lambda}$ 
induces a one-parameter group of 
automorphisms on $\Gi$, again denoted by $\delta_{\lambda}$. 
A function $f: \Gi \rightarrow \R$ is called homogeneous of degree $\alpha$ if
$f(\delta_\lambda(x)) = \lambda^\alpha f(x)$  for any $ \lambda>0$ and $x\in \Gi$. 
In particular we can define a 
gauge distance  $d(\cdot,\cdot)$ homogeneous of degree 1, 
as the image of the norm through the function $\Theta$:
\begin{equation}\label{2.5bis}
 d(y, \xn):= \|\Theta_{X_1,\cdots,X_n,y}(\xn)\|.\end{equation} 
The gauge function is  homogeneous of order 
\begin{equation}\label{homodim}
Q:= \sum_{i=1}^\kappa i \, dim(V^i) \end{equation}
with respect to the dilation. 
Hence $Q$ is called the homogeneous dimension of the space and 
there exist constants $C_1, C_2$ such that 
$$
C_1 r^Q \leq |B(x, r)|\leq C_2 r^Q \qquad\forall\, r>0, \ x\in \Gi,$$
where $B(\xn, r)$ denotes  the metric ball centered in $x$ with radius $r$, 
and  $|\,\cdot\,|$ denotes the Lebesgue measure. 

Any vector field $X$ will be identified with 
the first order differential operator with its same coefficients. 
If $\phi$ is a continuous function defined in an open set $V$ of $\Gi$ and if, for every $i=1,\cdots, m,$ there exists the Lie derivative 
$X_{i}\phi$ then we  call horizontal gradient of $\phi$ the vector
\begin{equation}\label{e:tutto subriemannian}
	\nabla \phi=\sum_{i=1}^m (X_{i}\phi) X_{i}.
\end{equation} 
The associated classes of H\"older continuous functions 
will be defined as follows:
\begin{definition}\label{defholder}
Let $0<\alpha < 1$, $V\subset\Gi$ be an open set, and  $u$ be a function defined on
$V.$ We say that $u \in C^{\alpha}(V)$ if there exists a positive constant $M$ such that for
every $x, x_0\in V$ 
\begin{equation}\label{e301}
   |u(x) - u(x_0)| \le M   d ^\alpha(x, x_{0}).
\end{equation}
We put  
 $$\|u\|_{C^{\alpha}(V)}=\sup_{x\neq x_{0}} \frac {|u(x) - u(x_{0})|}{d^\alpha(x, x_{0})}+ 
\sup_{V} |u|.$$
 Iterating this definition, 
 if $k\geq 1$  we say that $u \in
C^{k,\alpha}(V)$  if 
 $X_iu \in C^{k-1,\alpha}(V)$ for all $i=1,\cdots, m$.
\end{definition}
The Laplace type operator defined in \eqref{laplacoperator}
is a differential operator of degree $2$, in the sense of the following definition.
\begin{definition} Let $\{X_{i_j}\}$ be differential operators of order 1 and degree  $ deg(X_{i_j})$. 
We say that the  differential  operator $Y_1= X_{i_1}\cdots  X_{i_p}$ has order $p$ and degree 
$\sum_{j=1}^p deg(X_{i_j})$. 
Moreover, if 
 $Y$ is a differential operator  represented as 
\begin{equation}\label{order} 
Y=a Y_1,\end{equation}
where $a$ is a homogeneous function of degree $\alpha$, then 
we say that $Y$ 
is homogeneous of degree $deg(Y_1)-\alpha$. A differential operator will be called of degree  $k-\alpha$ if it is a sum of operators with maximum degree  
$k-\alpha$.
\end{definition}
%

 Following \cite{Folland} we recall the definition of kernel of type $\alpha$:
\begin{definition}\label{kerneltype}
We say that $K$ is a kernel of type $\alpha$, if $K$ is  of class $C^\infty$  away from $0$ and it is homogeneous of degree $\alpha-Q$.
\end{definition}
In a Carnot group, this implies that $K$ satisfies condition \eqref{e:sileva}.

\subsection{The Riemannian approximation of the structure}
One of the key technical instruments that we will use is a 
Riemannian approximation of the subriemannian structure.

For every $\e>0$, we extend the Riemannian metric defined on $V^1$ 
to a left invariant Riemannian metric
defined on $\mathfrak g$
by requesting that   \begin{equation}\label{e:pippo}(X_{1, \e},\cdots,X_{n, \e}):=
(X_1,\cdots,X_m,\e^{{deg}(m+1)-1}
X_{m+1},\cdots, \e^{{deg}(n)-1} X_n)\end{equation}
 is an orthonormal frame. We say that these vector fields {have $\e$-degree}
 equal to $1$, and we  write    $deg_\e(X_{i, \e})=1$.
Since the Lie algebra generated by these vectors also
contains the commutators of these vector fields, we also consider the vector fields
 \begin{equation}\label{e:pippo1}X_{i, \e}:=
X_{i -n+m} \  \text{and } \ {deg}_{\e}(X_{i, \e}):= {deg}(X_{i-n+m}) \ \text{for all}\ i= n+1,\cdots, 2n-m.\end{equation}

Let $d_{cc}$ and $d_{cc, \e}$	denote the  control distances 
 associated with the vector fields $X_{1},\cdots, X_{m}$ and $X_{1, \e},\cdots,X_{n, \e}$, respectively. It is well known (see for instance \cite{Gromov} and the references therein)  that $(\mathbb{G},d_{cc, \e})$  converges
in the Gromov-Hausdorff sense, as $\e\to 0$, to the subriemannian
space $(\mathbb{G},d_{cc})$. Although the subriemannian structure is formally recovered 
in the limit for $\e\rightarrow 0$,  we will need to recognize that 
the structure and all constants appearing in the estimates are stable in the limit. 
In addition we will need to recognize that the space has a property of  $\e$-homogeneity, 
with respect to the natural distance. 

A classical estimate of the distance $d_{cc, \e}$ is due to Nagel, Stein and Wainger in 
 \cite{NSW}.  From the whole family $\{X_{i, \e}\}_{i=1,\cdots, 2n-m}$ it is possible to select different bases  $\{X_{i_j, \e}\}_{i_j\in I}$, for different choices of indices $I = (i_1, \cdots,i_n) \subset \{1,...,2n-m\}^n$. As a consequence each vector 
$v$ has different representations $v= \sum_{i_j\in I} v_{i_j, \e}X_{i_j, \e}$ in terms of the 
different bases. The optimal choice of indices, denoted $I_{v, \e},$ 
is the one which minimize the $\e$-homogeneous gauge distance: 
\begin{equation}\label{rap}\|v\|_\e = \sum_{i_j\in I_{v, \e}} |v_{i_j, \e}|^{1/ deg_\e(i_j)}=\min_{I}\sum_{i_j\in I}|v_{i_j, \e}|^{1/deg_\e(i_j)}.
\end{equation}
This norm can be explicitly written as follows: 
if $v = \Theta_{X_1,\cdots, X_n, y}(x)$ then
\begin{equation}\label{norma_e}\|v\|_{\e}=\sum_{i=1}^{m}|v_i|   +  \sum_{i=m+1}^{n} \min\left\{\frac{|v_i|}{\e^{deg(i)-1}}, 
|v_i|^{1/deg(i)} \right\}.
\end{equation}
In \cite{CCR} it is proved that the associated ball box distance 
\begin{equation}\label{de}d_\e(y,\xn)= \|\Theta_{X_{1},\cdots, X_{n}, y}(\xn)\|_\e.\end{equation}
is locally equivalent to the distance $d_{cc,\e}$, with equivalence constants independent of $\e$.
Let us explicitly note that this distance has different behavior in $0$ and at infinity. Indeed, 
if  $v_i$  are small with respect to $\e$ for every $i\ge m+1$, then 
the distance $d_{\e}$ has a Riemannian behavior, while it is purely  subriemannian for $v_i$ large.

 It is worthwhile to note that for every $\e>0$ there exists a constant $C_\e$ such 
	$|B_\e(\xn, r)|= C_\e r^n$, 
where 
$B_\e(\xn, r)$ denotes the ball $\{\yn\in \mathbb{G}\,|\, d_\e(\xn, \yn)<r\}$ and 
$|\cdot|$ the Lebesgue measure.  In particular for every $\e>0$  the homogeneous dimension of the Riemannian space 
	is $n$, while by \eqref{homodim} for $\e=0$ the homogeneous dimension of the space is $Q$, with $Q>n$. 
Hence this notion of homogeneity is not 
	stable with respect to $\e$, and the constant $C_\e$ blows up for $\e\rightarrow 0$.  
However it has recently proved in \cite{CCR}
the following uniform doubling property: \begin{prop}\label{homog-stab}
There is a constant $C $ independent of $\e$ such that for every $\xn\in \mathbb{G}$ and $r>0$,
\begin{equation}\label{E:homog-stab}
	|B_\e(\xn, 2r)|\le C |B_\e(\xn, r)|.
\end{equation}
 \end{prop}
The doubling inequality \eqref{E:homog-stab} can be considered as a weak form of homogeneity, and suggests that it is possible to give a new definition of $\e$- homogeneity. 
Following \cite{RS} we will give the following definition of local homogeneous 
functions and operators
\begin{definition} \label{egrado}A function $f$ is locally  homogeneous of $\e$-degree 
$\alpha$ in a neighborhood of a point $z$ with respect to the metric \eqref{de} 
if $f\circ {\Theta_{X_1,\cdots,X_n,  z}^{-1}}$ is homogeneous of 
degree $\alpha$, with respect to the norm $\|\cdot\|_{\e}$  defined in \eqref{norma_e}.
A differential operator $Y$ is homogeneous of local  $\e$-degree $\alpha$ in a neighborhood of a point $z$, 
with respect to the metric \eqref{de} if 
$d\Theta_{X_1,\cdots, X_n,  z}(Y)$ is homogeneous of degree $\alpha$. \end{definition}

In particular this definition implies the following property:
\begin{remark}\label{edegreefo}
If $a$ is a homogeneous function of $\e$-degree $\alpha$ and $Y_1$ is an operator of $\e$-degree $k$, 
then $a(z^{-1}x)Y_1 $
is a homogeneous operator of $\e$-degree $k-\alpha$. This implies that for every other smooth function $f$ 
of local $\e$-degree $\beta$ in a neighborhood of a point $z$, 
$a(z^{-1} x)(Y_1 f)$
is a smooth function 
of $\e$-degree $\beta+\alpha-k $ in a neighborhood of a point $z$, and 
$|a(z^{-1} x)(Y_1 f)(x)|\leq C d_\e^{\beta+\alpha-k}(x,z)$.
\end{remark}

If $\phi\in C^{\infty}(\Gi)$ we   define the $\e$-gradient of $\phi$ as follows
\[\nabla_\e \phi:=\sum_{i=1}^n (X_{i, \e}\phi) X_{i, \e}.\]
%
%
In terms of the vector fields with  $\e$-degree $1$,   defined in \eqref{e:pippo}, we consider the associated heat operator
  \begin{equation}\label{operatore}
    L_{\e}:= \p_t - \sum_{i=1}^n  X^2_{i, \e} - \sum_{i=1}^m  b_i X_{i, \e}.  \end{equation}
    In analogy with the operator introduced in \eqref{e:tutto subriemannian},
the heat operator associated to the subriemannian structure has the form
  \begin{equation}\label{operatoresub}L:= \p_t - \sum_{i=1}^m  X^2_{i}- \sum_{i=1}^m b_i X_{i}.   \end{equation}
The behavior of these operators in interior points 
is well known:  they admit fundamental solutions respectively 
$\Gamma_{\e}(\xn, t)$ and  $\Gamma(x,t)$ of class $C^\infty$ out of the pole
(see \cite{JSC} for precise estimates of $\Gamma(x,t)$ and \cite{CCM} for estimates of 
$\Gamma_{\e}(\xn, t)$ uniform in $\e$). 

In our work we will 
need estimates which are uniform in the variable $\e$. We start with the following definition.
\begin{definition}\label{defkernelestimates}
We say that a family of kernels $(P_{\e})_{\e>0}$, defined on 
$\Gi\times  ]0, \infty[\times \Gi\times ]0, \infty[ $ 
and $C^\infty$ out of the diagonal, is of uniform exponential $\e$-type $\lambda +2$,  
if for $q\in \N$ and $k$-tuple $(i_1,\cdots,i_k)\in \{1,\cdots,n\}^k$ there exists $C_{q,k}>0$, depending only on $k,q$ and  on 
the Riemannian metric, such that
\begin{equation}\label{e:Giovanna}
|(X_{i_1, \e}\cdots X_{i_k, \e}\p_t^q P_{\e})((\xn,t), (z,\tau))| \le C_{q,k} (t-\tau)^{-q-k/2 + \lambda/2} 
\frac{e^{-\frac{d_\e(\xn, z)^2}{C_{q,k} (t-\tau)}}} {|B_\e(x, \sqrt{t-\tau})|}
\end{equation}
for all $x\in \Gi$ and $t>\tau$. 
\end{definition}

According with the definition above, the fundamental solution $\Gamma_{\varepsilon}$ is a kernel of exponential
$\e$-type $2$. Precisely,  the following result, established in \cite{CC} (see also \cite{CM}  and \cite{CCM}), holds:

\begin{theorem}\label{uniform heat kernel estimates-e}
The fundamental solutions $\Gamma_\e(x,t)$ of the operators $L_\e$ constitutes 
a family of  kernels of exponential
$\e$-type $2$ and there exist constants $C_0,C>0$  
independent of $\e$ such that for each $\e>0$, $x\in \mathbb{G}$ and $t>\tau$ one has
\begin{equation}\label{gauss-sol}
C_0^{-1}\frac{e^{-C\frac{d_\e(\xn, z)^2}{(t-\tau)}}} {|B_\e (x, \sqrt{t-\tau})|}\le \Gamma_{\e}((\xn, t), (z, \tau))\le
C_0\frac{e^{-\frac{d_\e(\xn,z)^2}{C (t-\tau)}}} {|B_\e (x, \sqrt{t-\tau})|}.
\end{equation}
Moreover, for any 
 $k$-tuple $(i_1,\cdots,i_k)\in \{1,\cdots, m\}^k$  one has 
\begin{equation}\label{GetobarG}{X}_{i_1}\cdots {X}_{i_k} \p_t^q  \Gamma_{\e}\to  {X}_{i_1}\cdots {X}_{i_k}\p_t^q \Gamma\quad \text{as $\e\to 0$}\end{equation}
uniformly on compact sets and  in a dominated way on all $\Gi$. 
\end{theorem}

\begin{remark}
In particular from this theorem we can obtain the well known Gaussian estimates 
of the fundamental solution $\Gamma$ of the operator $L$. Indeed $\Gamma$ is a kernel of exponential
type $2$ and there exist constants $C_0,C>0$  such that for each $x\in \mathbb{G}$ and $t>\tau$ one has
\begin{equation}\label{gauss-sol2}
C_0^{-1}\frac{e^{-C\frac{d(\xn, z)^2}{(t-\tau)}}} {|B (x, \sqrt{t-\tau})|}\le \Gamma((\xn, t), (z, \tau))\le
C_0\frac{e^{-\frac{d(\xn,z)^2}{C (t-\tau)}}} {|B (x, \sqrt{t-\tau})|}.
\end{equation}
\end{remark}

\subsection{The parametrix method}\label{paramethod}
One of the main instruments that we will use to estimate the fundamental solution is the parametrix method, originally due to Levi 
and now extremely classical for elliptic and parabolic equations 
(see \cite{Fri}). In subriemannian setting the parametrix method 
have been used to approximate general H\"ormander type operators with homogeneous ones: we refer to \cite{RS, SC} for the first results, \cite{JSC} for the subriemannian heat kernel, \cite{C} for estimates in case of low regularity, \cite{blu_a} 
for a recent self-contained presentation. 
The method consists in 
providing an explicit representation of the fundamental solution $\Gamma$ of an operator 
$L$ in terms of the fundamental solution of an approximating operator 
$L_{z_1}$ (with associated fundamental solution $\Gamma_{z_1}$). 
Using  the definition of fundamental solution and the fact that 
$L_{z_1} (\Gamma -  \Gamma_{z_1}) =(L_{z_1} -L)\Gamma,$
the difference between the 
two solutions can be formally represented as
\begin{align}\nonumber (\Gamma -  \Gamma_{z_1})((\xn, t), (z, \tau))&=
 \int \Gamma_{z_1}((\xn, t), (y, \theta)) (L_{z_1} - L)(\Gamma -  \Gamma_{z_1}) ((y, \theta), (z, \tau))\,dy\,d\theta \\ 
&+\int \Gamma_{z_1}((\xn, t), (y, \theta)) (L_{z_1} - L)  \Gamma_{z_1} ((y, \theta), (z, \tau))\,dy\,d\theta.\label{e:acca}\end{align}
Denoting by 
\begin{equation}\label{piantina}H: =  L_{z_1} - L,\end{equation}
and $E_{\Gamma_{z_1}}$  the integral operator with kernel $\Gamma_{z_1}$,
the above expression \eqref{e:acca} can be written as 
$$(I- E_{\Gamma_{z_1}} H)(\Gamma -  \Gamma_{z_1}) = E_{\Gamma_{z_1}}H(\Gamma_{z_1}).$$
If the operator $(I- E_{\Gamma_{z_1}} H)$ is invertible, the difference $\Gamma -  \Gamma_{z_1}$ can be formally expressed as
\begin{equation}\label{e: informal serie}
\Gamma -  \Gamma_{z_1}=\sum_{j=0}^\infty (E_{\Gamma_{z_1}} H)^{j+1}(\Gamma_{z_1})= E_{\Gamma_{z_1}} \Phi, \quad \text{with}
\;\; \Phi:=\sum_{j=0}^\infty (H E_{\Gamma_{z_1}})^{j}H (\Gamma_{z_1}).
\end{equation}
Roughly speaking the proof is obtained as follows:
\begin{itemize}
\item [1)] The first  and most delicate part of the proof is to define the approximating operator $H$
and to prove that it is a differential operator of degree $2-\alpha$ for a suitable positive $\alpha$. From this fact it 
follows that the kernel 
\begin{equation}\label{erreuno}R_1(x,z): = H \Gamma_{z_1}(x,z)
\end{equation} is homogeneous of type $\alpha$ with respect to the considered homogeneous space. It is important to note that 
$$R_1(x,z)= (L_{z_1}-L)\Gamma_{z_1}(x,z) = L(\Gamma(x,z) - \Gamma_{z_1}(x,z))\,.$$
\item [2)]  Identifying 
the operator $HE_{\Gamma_{z_1}}$ with the integral 
operator $E_{R_1}$ with kernel $R_1$, 
the series $\Phi$ in \eqref{e: informal serie}
reduces to 
\begin{equation}\label{Fi}\Phi=\sum_{j=0}^\infty (E_{R_1})^{j}E_{R_1}.\end{equation} 
Using the fact that the convolution of a kernel of type $\alpha$ with a kernel of type $\beta$ 
provides a kernel of type $\alpha + \beta$, it is possible to prove that this series
converges uniformly (see for example Lemma 7.3 in \cite{Jerison}). 
\item [3)]  Finally, singular integrals tools  lead to the convergence of the derivatives and the function $\Gamma$, 
defined through \eqref{e: informal serie}, is a fundamental solution. 
\end{itemize}
In the sequel we will consider kernels of type $\alpha$ in the sense of 
Definition \ref{kerneltype}, when working with subelliptic operators, 
while kernels of exponential $\e$-type  
$\alpha$ in the sense of Definition \ref{defkernelestimates}, when studying Riemannian heat kernels. The main difficulty to be faced here is that the Riemannian approximation has not 
a standard notion of homogeneity, since the Riemannian homogeneous dimension $n$ 
collapses to the subriemannian one $Q$ in the limit. However
we have endowed the regularized space with an $\e$-homogeneous structure (see \eqref{E:homog-stab} and the remark below) and
 we will see that this is enough to apply the method in this setting. 
Therefore, even though our vector fields are homogeneous, 
our approach is more similar to the ones 
\cite{RS, SC, JSC, C} where the geometry of the given operator and the 
approximating one do not coincide.

\section{Reproducing formula on a plane}\label{s:LaplaceBeltrami}
In this section we will prove a first version of Theorem \ref{teorema1}, 
under the simplified assumptions that $\partial D$ is a non characteristic plane and 
that the vector fields $(X_i)_{i=1, \cdots, m}$ are the generators of a Carnot group and have the 
explicit representation recalled in \eqref{struttura campi}. 
This result will be obtained via a parabolic approximation and a Riemannian regularization. 
The proof of the same Theorem \ref{teorema1} on any smooth hypersurface will be deduced from this result in next section.

\subsection{Geometry of the plane} \label{geosection}
Let  $\mathbb{G}$ be a Carnot group of step $\kappa$. 
Consider a  non characteristic plane $M_0$. 
Using the logarithmic coordinates defined in \eqref{deftheta}, 
it is always possible to represent $M_0$ as follows: 
\begin{equation}\label{M0}M_{0}=\{(x_1,\x)\in \mathbb{G}:\,x_1=0\},
\end{equation}
 where  
$x=(x_1, \x)$ is a point of the space,  $x_1\in \R$ and $\x\in \R^{n-1}$.
This choice of coordinates is made in such a way that 
the vector fields $X_{1} = \partial_1$ coincides with the direction normal to the plane, 
while $\{X_i\}_{i=2, \cdots, n}$  are tangent to  $M_0$  and 
are represented as in \eqref{struttura campi}. Hence the vector fields obtained 
from $X_i$ by evaluating the coefficients $a_{ij}$ on the points of the plane $M_0$, 
are the generators of the first layer $\hat V_1$ on the plane, 
so that 
\begin{equation}
\label{campitildeu_noe}  \hat X_{i}:= {X_{i}}_{|x_1 =0},\quad i=2,\cdots,n. \end{equation}
Thanks to this choice, not only the plane $M_0$ is non characteristic, 
but also the planes $M_{z_1}=\{(x_1,\x)\in \mathbb{G}:\,x_1=z_1\}$, 
for every $z_1$ sufficiently small, are non characteristic.  
We note also that assumption \eqref{assumption}
ensures that the vector fields $\hat X_i$ 
 satisfy a H\"ormander condition at every point and span a $n-1$ dimensional space at every point. In analogy with formula \eqref{homodim}, 
the homogeneous dimension of the plane is $$\hat Q= \sum_{i=2}^\kappa i \, dim(\hat V^i).
$$
As a consequence 
\begin{equation}\label{hadimension}\hat Q= Q-1.\end{equation}
Via the exponential map and definition \eqref{2.5bis} 
the vector fields  $\hat X_{i}$  define a distance   
\begin{equation}\label{hatdistance}
\hat d(\hat y, \hat x):= \|\Theta_{\hat X_2,\cdots, \hat X_n ,\y}(\x)\| 
\end{equation}
  on  $M_0.$ By the H\"ormander condition a
Laplace operator and its time dependent counterpart are defined on $M_{0}$ as 
\begin{equation}
\hat \Delta:= \sum_{i=2}^m \hat X_{i}^2, \quad  \text{and} \quad
\hat L:=\partial_t -\hat \Delta,
\label{heat_plane_noe}
\end{equation}
and they have non negative fundamental solutions
 $\hat \Gamma_{\hat\Delta}$ and $\hat \Gamma$ respectively. 

In analogy with Definition \ref{defkernelestimates} we give the following 
\begin{definition} We
say that a kernel $\hat P$, defined on 
$\R^{n-1}\times  ]0, \infty[\times \R^{n-1}\times ]0, \infty[ $ 
and $C^\infty$ out of the diagonal, is of exponential type $\lambda +2$,  
if for $q\in \N$ and $k$-tuple $(i_1,\cdots,i_k)\in \{1,\cdots,n\}^k$ there exists $C_{q,k}>0$, depending only on $k,q$ and  on 
the subriemannian metric, such that
\begin{equation*}
|(\hat X_{i_1}\cdots \hat X_{i_k}\p_t^q \hat P)((\x,t), (\hat z,\tau))| \le C_{q,k} (t-\tau)^{-q-k/2 + \lambda/2} 
\frac{e^{-\frac{\hat d(\x, \hat z)^2}{C_{q,k} (t-\tau)}}} {|\hat B(\hat x, \sqrt{t-\tau})|}
\end{equation*}
for all $\hat x\in \R^{n-1}$ and $t>\tau$. 
\end{definition}
Our first result is the following one:
\begin{theorem}
\label{lemmaJe1}
Assume that $M_{0}=\{(x_1,\x)\in \mathbb{G}:\,x_1=0\}$ is a non characteristic plane 
and let $T>0$. 
Then there  exists  a constant $C=C(T)$  such that 
 for 
all  $z=(0, \z)$, $x=(0, \x)$ in $M_0$ and for every $t$ and $\tau$, with $0<t-\tau<T$, we have 
 \begin{equation}\label{goal}|\hat \Gamma((\x,t), (\z, \tau)) -\sqrt{t-\tau}\Gn((x, t), (z, \tau))|\le   
 C \hat \Gn((\hat x, t), (\hat z,\tau))(t-\tau)^{1/4}.\end{equation}
In addition $\hat \Gamma((\x,t), (\z, \tau)) -\sqrt{t-\tau}\Gn((x, t), (z, \tau))$ is an operator of exponential type $1/4$ with respect to the 
vector fields $\{\hat X_i\}_{i=2, \cdots, m}$.
 \end{theorem}
This theorem will be proved with the parametrix method and a Riemannian approximation.  
Classically, the method is  applied for proving the existence of the  fundamental solution 
of a given operator. Extending an approach of \cite{CCS}, in Lemma \ref{lemmaJe1} we apply the method to find a relation 
between the fundamental solutions since we already know that they exist. 

Even though the parametrix method has  been largely used in the subriemannian setting 
for internal estimates, the vector fields $\hat X_{i}$ 
do not provide a subriemannian approximation of the vector fields $X_{i}$
 and the standard parametrix method of Rothschild and Stein can not be applied starting with the 
fundamental solution of  $\hat L$.
 In order to clarify this fact we  start with the concrete example of vector fields already introduced in 
\eqref{campesempio}. 
\begin{example}
Let us consider  the vector fields in \eqref{campesempio}. Their commutators  
are $\partial_4 = [X_1, X_2]$, which is a vector field of degree $2$, and 
$\partial_5 = [[X_1, X_2],X_2]$, which is a vector field of degree $3$. 

If we evaluate the vector fields $X_i$ on the plane $\{x_1 =0\}$ 
we obtain 
$$\hat X_{2} = \partial_2 + x_3 \partial_4, 
\quad \hat X_{3} = \partial_3 + x_4 \partial_5,
$$
so that 
$$X_{2} - \hat X_{2}= x_1^2 \partial_5 \ \text{is an operator of degree }1.$$
Consequently, the difference 
$$H=\hat L-L\, \text{ is an operator of degree } 2.$$
 Hence it is not possible to apply the parametrix method, whose convergence requires 
 $H$ to be an operator of degree strictly less than $2$. 
\end{example}

Due to these difficulties, we introduce a new version of the 
parametrix method, using the $\e$-Riemannian approximation described in Section \ref{section2}. 
The whole proof is based on a careful analysis of the Riemannian approximation 
metric and lies on a delicate interplay between the Riemannian and subriemannian nature of our operators. 

\subsection{A Riemannian and frozen approximating operator}
In Section \ref{geosection} we have chosen a point $0$ 
and a constant $\varepsilon$ sufficiently small such that for every $z_1\in\R$, such that $|z_1|\le \e^{2\kappa}$, the plane 
$M_{z_1}=\{(x_1,\x)\in \mathbb{G}:\,x_1=z_1\}$ 
is non-characteristic.
%
In analogy with \eqref{campitildeu_noe}, we define the vector fields 
$X_{i, z_1}:= {X_{i}}_{|_{x_1 =z_1}}$ as the vector fields whose coefficients 
are evaluated at the points with first component $z_1$. Thanks to \eqref{struttura campi}, 
they can be represented as
\begin{equation}
\label{campizetauno}
X_{1, z_1}:=\partial_1,\quad   X_{i, z_1}:= {X_{i}}_{|_{x_1 =z_1}}= \partial_i + \sum_{deg(j) > deg(i)} a_{ij}(z_1, \x) \partial_j
,\quad i=2, \cdots, n
\end{equation}
for every $\e>0$ and  we set 
\begin{equation}
\label{campitildeu}  
X_{1, z_1, \e}:=\partial_1,\quad    X_{i, z_1, \e}:= {X_{i, \e}}_{|_{x_1 =z_1}},\quad i=2, \cdots, 2n-m.\end{equation}
We  introduce now an 
operator $L_{z_1, \e}$ 
which can be split in a tangential and in a normal part on any 
plane $M_{z_1}$, and  we will use it to approximate with the parametrix 
method the tangential 
and the normal part of the operator $L_\e$.
Precisely, 
we define 
\begin{equation}
L_{z_1, \e}
:= \partial_t - \sum_{i=1}^n X_{i, z_1, \e}^2
\label{L0riemannian}
\end{equation}
with fundamental solution $\Gamma_{z_1, \e}$
on the whole space. On every plane $M_{z_1}$ we define the tangential operators 
\begin{equation}
\hat L_{z_1, \e}:=\partial_t -\hat \Delta_{z_1, \e},\,
\quad \text{ where } \quad 
\hat \Delta_{z_1, \e}:= \sum_{i=2}^n X_{i, z_1, \e}^2, \quad  
\label{heat_plane}
\end{equation}
with non negative fundamental solutions $\hat \Gamma_{z_1, \e}$ and $\hat \Gamma_{\Delta, z_1, \e}$ respectively. 
%
\begin{remark}
\label{r:splittingNC} 
Let us explicitly note that the 
operator $L_{z_1, \e}$ can be represented as 
$$L_{z_1, \e}
= \partial_t - \partial^2_{11} - \hat \Delta_{z_1, \e}.$$
Since $\partial_{1}$ coincides with the 
direction normal to the plane, the operator 
splits in the sum of its orthogonal part 
$\partial_t - \partial^2_{11}$ and its tangential part $\hat L_{z_1, \e}$. 
Consequently its fundamental solution can be represented as 
$$ \Gamma_{z_1,\e} =    \Gamma_{\perp,z_1,\e}  \hat \Gamma_{z_1,\e}.$$
where $ \hat \Gamma_{z_1,\e}$ is defined above and $ \Gamma_{\perp,z_1,\e} $ is 
the standard one-dimensional Gaussian function, fundamental solution of $\partial_t - \partial^2_{11}$. 
\end{remark}

\subsection{Estimates of the approximating operator}
As recalled in Section \ref{paramethod} the first step of the 
parametrix method is to prove that the difference  
$X_{i,\e}- X_{i,z_1,\e} $ is a differential operator of $\e$-degree strictly less than $1$ 
around the point $z_1$ 
and, as a consequence, that the operator $H_\e: =L_\e-L_{z_1, \e}$ (see also \eqref{operatore} and \eqref{piantina} above)
  has $\e$-degree strictly less than 2. 
\begin{lemma}
\label{lemmallemma} 
Let  $M_{z_1}=\{(x_1, \x)\in\mathbb{G}:x_1 =z_1\} $ be a non characteristic plane.
For every $z=(z_1, \hat z)\in M_{z_1}$, and for every $i\leq n$ and for every $h$ such that 
$deg(i) +1 \leq deg(h) \leq \kappa$ there exists a polynomial $b_{i,h, z_1}(v)$, 
homogeneous of degree $deg(h)-deg(i)-1$ as a function of $v$ and $z_1$, such that, if $v = 
\Theta_{X_{z_1}, z}(x)$, 
\begin{equation} d\Theta_{X_{z_1}, z}({X_{i}}-{X_{i,z_1}})=
v_1\sum_{deg(h)=deg(i)+1}^\kappa b_{i,h, z_1}(v)d\Theta_{X_{z_1}, z}(X_{h,z_1}),\end{equation}
where $\Theta_{X_{z_1}, z}$ has been defined in \eqref{deftheta}. 
Moreover
\begin{equation}\label{homihz} |b_{i, h, z_1}(v)|\leq C \sum_{j=0}^{deg(h)-deg(i)-1} |z_1|^j \|v\|^{deg(h)-deg(i)-1-j}.
\end{equation}
\end{lemma}
\begin{proof} 
When $i=1,\cdots, n$,  by the definition  
\eqref{struttura campi} and \eqref{campizetauno} of the vector fields 
we have, for $deg(i)=\kappa$
\begin{equation}\label{eerre}X_{i}- 
X_{i,z_1}= 0,
\end{equation}
hence the thesis is true, and we have to prove it only for $deg(i)<\kappa$.
Using the fact that the translation associated to the 
vector fields $X_{z_1}$ 
acts only on the $\x$ variables, 
we have 
\begin{align}\nonumber &
d\Theta_{X_{z_1}, z}(X_{i}- X_{i,z_1})  = \sum_{deg(h)=deg(i)+1}^\kappa
\Big(   a_{i, h}(x_1, \hat v) - 
a_{i, h}(z_1,\hat v)
\Big) \partial_{h}  =\\& = \sum_{deg(h)=deg(i)+1}^\kappa (x_1-z_1) a^1_{i, h, z_1}(v) \partial_{h}=v_1 \sum_{deg(h)=deg(i)+1}^\kappa  a^1_{i, h, z_1}(v) \partial_{h}.
 \label{Xiuestimatebis}\end{align}
In the last equality we have denoted $(x_1-z_1)a^1_{i, h, z_1}(v) $ the polynomial 
$ a_{i, h}(x_1, \hat v) - 
a_{i, h}(z_1,\hat v)$ and we used the fact that $v_1 = x_1-z_1$. The polynomial $a^1_{i, h, z_1}(v)$  is homogeneous of degree 
$deg(h)-deg(i)-1$ in the variables $v$ and $z_1$ and we have estimated as 
$$ |a^1_{i, h, z_1}(v)|\leq C \sum_{j=0}^{deg(h)-deg(i)-1} |z_1|^j \|v\|^{deg(h)-deg(i)-1-j}.$$
If $deg(i)= \kappa-1$, the proof is completed, by \eqref{eerre}. 
For   $deg(i)< \kappa-1$, using again the expression \eqref{campizetauno} for 
$deg(h)<\kappa$ and \eqref{eerre} for $deg(h)=\kappa$, we get 
$$d\Theta_{X_{z_1}, z} (X_{i}- X_{i, z_1}) = v_1 \sum_{deg(h)=deg(i)+1}^\kappa a^1_{i, h, z_1}(v) 
d\Theta_{X_{z_1}, z}(X_{h, z_1}) -
$$
$$-v_1\sum_{deg(j)=deg(i)+2}^\kappa
\sum_{deg(h)=deg(i)+1 }^{deg(j)-1}a^1_{i, h, z_1}(v)a_{h, j, z}(v) \partial_{v_j}.$$
Since the Lie group is nilpotent of step $\kappa$, 
 after $\kappa-1$ iteration of this method, we get that there exists a polynomial $b_{i,h,z_1}$
such that \eqref{euno} is satisfied. 
 \end{proof}

From this lemma, Corollary \ref{7-1} below immediately follows. The proof is technically very simple, 
but it is important to note that the $X_{i,\e}-X_{i,z_1,\e}$ is a differential operator which has 
local degree $1$ while has local $\e$-degree $1/2$, in a neighborhood of the point $z$. 
This property allows to obtain a better approximation in the 
Riemannian setting, rather than in the subriemannian setting. 

\begin{corollary}
\label{7-1}
Let $M_0$ be a non characteristic plane.  
Let $\mathcal{S}$ be the strip 
 \[\mathcal{S}:=\{x=(x_1, \x)\in\mathbb{G}:\,|x_1|\leq \e^{2\kappa}, \ 
 |x_1-z_1| \leq \e^{2\kappa}\},\]
where $\kappa$ is the step of the group. 
Then $X_{i,\e}-X_{i,z_1,\e}$ is a differential operator of 
{{$\e$-degree $1/2$ with respect to the vector fields $X_{i,z_1,\e}$.}}
\end{corollary}
\begin{proof}
Applying Lemma \ref{lemmallemma}, calling 
$p_{i,h, z_1, \e}(v) = \e^{-deg(h)} v_1 b_{i,h, z_1}(v)$ 
and using the 
fact that $|v_1|\leq\e^{\kappa}$ and $|z_1|\le 1$ we have
$$|p_{i,h, z_1, \e}(v)| \leq C  |v_1|^{1/2} \sum_{deg(h) = deg(i) +1}^\kappa \sum_{j=0}^{deg(h)-deg(i)-1}\|v\|^{deg(h)-deg(i)-1-j}. $$
 Since $X_{i,\e} = \e^{deg(i)} X_{i} $ and  $v_1  b_{i,h, z_1}(v)X_{h,z_1} = p_{i,h, z_1, \e}(v)X_{h,z_1, \e}$,  from Lemma  \ref{lemmallemma} we also deduce that
\begin{equation}\label{euno} d\Theta_{X_{z_1}, z}(X_{i,\e}-X_{i,z_1,\e})=
\e^{deg(i)} \sum_{deg(h)=deg(i)+1}^\kappa p_{i,h, z_1,\e}(v) d\Theta_{X_{z_1}, z}(X_{h,z_1, \e}).\end{equation}
If $\|v\|\leq 1$, then 
$|p_{i,h, z_1, \e}(v) |\leq C  |v_1|^{1/2}.$ Since 
$X_{h,z_1, \e}$ has degree $1$, then $p_{i,h, z_1,\e}(v) d\Theta_{X_{z_1}, z}(X_{h,z_1, \e})$ is a differential operator of 
local $\e$-degree $1/2$ in the set $\|v\|\leq 1$ with respect to the vector fields $X_{i,z_1,\e}$. 

On the other side, if $\|v\|\geq 1$, we have that 
$X_{h,z_1}$ is a differential operator of $\e$-degree $h$,  
$|p_{i,h, z_1, \e}(v)| \leq C  |v_1|^{1/2} \|v\|^{deg(h)}, $
so that $p_{i,h, z_1,\e}(v) d\Theta_{X_{z_1}, z}(X_{h,z_1, \e}) 
$ is a differential operator of 
$\e$-degree $1/2$ for $\|v\| \geq 1$ with respect to the vector fields $X_{i,z_1,\e}$. 
\end{proof}

 As a direct consequence of the previous corollary, we can prove that the 
difference $H_\e= L_\e - L_{z_1,\e}$ is a differential operator of degree strictly less than 2 
in a neighborhood of the point $z$. 
\begin{lemma}
\label{7-2}Under the assumptions of Corollary \ref{7-1},  
${L_{\e}}-{L_{z_1,\e}}$ is an operator of $\e$-degree $3/2$ with respect to the vector fields 
$X_{z_1, \e}$. 
Precisely there exist polynomials $p^{(1)}_{h,z_1,\varepsilon}$, $p^{(2)}_{i,h,z_1,\e}$  
 and a constant $C$  independent of $\e$ satisfying 
 $$|p^{(1)}_{h,z_1,\e}(v)|\leq C |v_1|^{1/2} \;\; \text{for } \|v\|\leq 1, \quad 
|p^{(1)}_{h,z_1,\e}(v)|\leq C |v_1|^{1/2}\|v\|^{deg(h)} \;\; \text{for } \|v\|\geq 1$$ 
and
 $$|p^{(2)}_{i,h,z_1,\e}(v)|\leq C |v_1|^{1/2} \;\; \text{for } \|v\|\leq 1, \quad 
|p^{(2)}_{i,h,z_1,\e}(v)|\leq C |v_1|^{1/2}\|v\|^{deg(h)+deg(i)} \;\; \text{for } \|v\|\geq 1,$$ 
 such that 
\begin{equation}\label{somma}
d\Theta_{X_{z_1}, z}({L_{\e}-L_{z_1,\e}})=
\sum_{deg(h)=2}^\kappa   p^{(1)}_{h, z_1, \e}(v) d\Theta_{X_{z_1}, z}(X_{h, z_1, \e}) \; +\end{equation}
$$+\sum_{deg(h)=2}^\kappa \sum_{deg(i)=1}^{deg(h)-1}p^{(2)}_{i,h,z_1,\e}(v) 
 d\Theta_{X_{z_1}, z}(X_{i, z_1, \e} X_{h, z_1, \e}).$$
\end{lemma}
\begin{proof}
By the definition of the operators we have:
\begin{equation}\label{lele}\begin{split}
& d\Theta_{X_{z_1}, z}({L_{\e}-L_{z_1,\e}}) \\&=
\sum_{i=1}^{n}   d\Theta_{X_{z_1}, z}\Big({X_{i, z_1, \e}}  {(X_{i,   \e} -  X_{i, z_1, \e})}\Big) +
\sum_{i=1}^{n}   \Big(d\Theta_{X_{z_1}, z} ({X_{i, \e}} -  {X_{i, z_1, \e}})\Big)^2+\\&
+ \sum_{i=1}^{n} d\Theta_{X_{z_1}, z}\Big(({X_{i,  \e}} - {X_{i, z_1, \e}})   
{X_{i, z_1, \e}}\Big) -\sum_{i=1}^{m} b_i   d\Theta_{X_{z_1}, z}(X_{i,  \e}) \end{split}	
\end{equation}
Let us consider  the first term at the right hand side. 
By Corollary \ref{7-1}
$$  d\Theta_{X_{z_1}, z}({X_{i, z_1, \e}})  d\Theta_{X_{z_1}, z}(X_{i,   \e} -  X_{i, z_1, \e}) = $$$$=
d\Theta_{X_{z_1}, z}(X_{i, z_1, \e})\Big(\sum_{deg(h)=deg(i)+1}^\kappa p_{i,h,z_1,\e}(v)d\Theta_{X_{z_1}, z}(X_{h,z_1,\e})\Big).
 $$
When the derivative is applied on the coefficients $p_{i,h,z_1,\e}$, 
we obtain a term  contributing to the terms $p^{(1)}_{h,z_1,\e}(v) d\Theta_{X_{z_1}, z}(X_{h, z_1, \e})$. 
When the derivative is applied on the differential operator we obtain a contribution 
to the second sum in the right hand side of \eqref{somma}.
The other terms of \eqref{lele}  can be handled in a similar way. 
\end{proof}

In analogy with \eqref{erreuno}
we   define the kernel 
\begin{equation}\label{formula}
R_{1,\e}((\xn, t), (z, \tau)):=(L_{{z_1,\e}}-L_{\e})\Gamma_{z_1,\e} ((\xn, t), (z,\tau)).
\end{equation}
for 
$t>\tau$. 

As a consequence of Lemma \ref{7-2} and of the homogeneity of 
the fundamental solution, we provide an estimate for $ R_{1,\e}$.
\begin{lemma} \label{lemmaR1} 
If $M_0\subset\Gi$ is a non-characteristic plane,  
 $R_{1,\e}$ is a family of kernels of $\e$-uniform exponential type $1/2$ 
in the set $\{|x_1|\leq \e^{2\kappa}\}$. 
 Precisely for every bounded set there exists a constant $C$ such that for every 
$x= (x_1, \x), \, z=(z_1,\z)\in \Gi$ such that $|x_1|, |z_1|\leq \e^{2\kappa}$.
\begin{equation}\label{claim1R1}
|R_{1,\e} ((x, t), (z, \tau))|
\leq C \frac{ \Gamma_{z_1,\e}((\xn, 2t), (z, 2\tau))}{|t-\tau|^{3/4}},
\end{equation}
with  $C$ independent of $\e$ and $z_1$.
\end{lemma}
 \begin{proof} 
By the representation formulas obtained in the previous Lemma \ref{7-2}, used with $v_1=x_1-z_1$, we only have to estimate terms of the type
$|x_1-z_1|^{1/2}X_{i,z_1,\e} X_{h,z_1,\e}\Gamma_{z_1,\e}((\xn, t), (z, \tau))$. 
Using \eqref{e:Giovanna} 
for any $0<\e<1$ we immediately obtain 
$$|R_{1,\e} ((x, t), (z, \tau))|
\leq C \frac{|x_1-z_1|^{1/2} \Gamma_{z_1,\e}((\xn, 2t), (z, 2\tau))}{|t-\tau|}.$$
 In order to prove 
\eqref{claim1R1} we note that we can assume that 
  $|x_1-z_1|> \sqrt{t-\tau}$, since otherwise the assertion is true. 
In this case we can use the fact that $\rho^{1/4} e^{-\rho^2}\le Ce^{-\rho^2/2}$, for a suitable constant $C$, and the estimate \eqref{gauss-sol} of the fundamental solution to ensure that 
\begin{equation}
 \frac{|x_1-z_1|^{1/2}}{|t-\tau|^{1/4}} \Gamma_{z_1,\e}((\xn, t), (z, \tau))\leq 
 C\Gamma_{z_1,\e}((\xn, 2t), (z, 2\tau)),  
\end{equation}
From here the thesis follows at once. 
 \end{proof}

\subsection{Convergence of the parametrix method}
The second step of the parametrix method consists in 
proving that the 
series $\Phi$, defined in \eqref{Fi}, is convergent. 
In order to do this, we first need to obtain an uniform estimate 
of the distances  $d_{z_1, \e}$. 
We denote respectively 
 $d_{z_1, \e} $ and $d_{z_1, 0} $  the distances defined 
as in \eqref{de} and \eqref{2.5bis} in terms of the vector fields 
$X_{i, z_1, \e}$ and $X_{i, z_1}$:
\begin{equation}\label{finalmente2}
d_{z_1, 0}(x,z) =  \|\Theta_{X_{1, z_1},\cdots, X_{n, z_1}, z}(\xn)\|, 
\,\, 
d_{z_1, \e}(x,z) =  \|\Theta_{X_{1, z_1, \e},\cdots, X_{n, z_1, \e}, z}(\xn)\|_\e.
\end{equation}

Under the usual assumption that $M_0\subset\Gi$ is a non characteristic plane (so that  also $M_{z_1}$ is non characteristic for $|z_1|$ sufficiently small),  
we have the following lemmata. 
\begin{lemma}\label{distanze-fuori}  
For every $x=(x_1,\x)$, $z=(z_1,\z)$ in $\Gi$,
\begin{equation}\label{equivd2}
 d(x,z) = d_{z_1, 0}(x,z), \quad 
 d_{\e}(x,z) = d_{z_1, \e}(x,z).
\end{equation}
In addition the distance $\hat d$ defined in Section \ref{geosection} satisfies
\begin{equation}\label{finalmente}\hat d (\x, \z) = d_{0, 0} ((0,\x), (0,\z)). 
\end{equation}
and 
\begin{equation}\label{equivdhat}
 \hat d(\hat x,\hat z) = d((0,\hat x),(0, \hat z)).
\end{equation}
\end{lemma}
\begin{proof}
 The distance $d(x,z)$  
is defined in \eqref{2.5bis} as the norm of the vector $v$ such that  
\begin{equation}\label{carciofi_arrosto}
x=\operatorname{exp}(v_1 X_1)\operatorname{exp}(\sum_{i=2}^nv_i X_i)(z).\end{equation}
Since all the vector fields $(X_i)_{i=2,\cdots, n}$ are tangential to the plane $M_{z_1}$, 
the integral curve $t\mapsto \operatorname{exp}(t\sum_{i=2}^nv_i X_i)(z)$ is tangent to the same plane. Therefore,  along this curve  
the vector fields  $(X_i)_{i=2,\cdots, n}$ are computed for $x_1=z_1$ and coincide with the vector fields $X_{i, z_1}$. This implies that   $d(x,z) = d_{z_1, 0}(x,z)$.
The same argument applies to the second equality in \eqref{equivd2} to \eqref{finalmente}, and to 
\eqref{finalmente}. 
\end{proof}

Since we have a good estimate of the kernel $R_{1, \e}$
only in an $\e$- neighborhood of the plane $M_0$, 
we have to modify the classical parametrix method, restricting the 
integral in this neighborhood. 
To this end we consider a cut-off function depending only on  the first variable $x_1$. 
 Precisely, we consider  a piecewise function $\rho_\e$, supported 
in an $\e$ neighborhood of the origin, defined as follows:
\begin{equation}\label{chiNC}
\rho_\e(x_1)  =1 \  {\rm if}\  |x_1|\le 2\e^{2\kappa}, \quad \rho_\e(x_1)  =0 \  {\rm elsewhere}.\end{equation}

For any suitable kernel $K$, we define 
\begin{equation}	\label{martedi}
E_{R_1, \e}(K)((\xn, t), (z, \tau)):=-\int_{\mathbb{R}^{n}\times [\tau,t]} 
R_{1,\e}((x,t),  (y, \theta))\,
K ((y, \theta), (z,\tau)) \rho_\e(y_1-z_1)\, dy d\theta
\end{equation}
and, in analogy with \eqref{Fi},  we consider 
\begin{equation}	\label{giovedi}\Phi_\e((x, t), (z,\tau)):=\sum_{j=0}^{\infty}(E_{R_1, \e})^j(R_{1,\e})((x, t), (z,\tau)).\end{equation}
We will  prove that the series 
can be is totally convergent on the set 
\[\left\{0<t-\tau\le T, \ |x_1|,|z_1|\le \e^{2\kappa}, \  d_{\e}(x, z)+
|t-\tau|^{\frac{1}{2}}\ge \delta
\right\}\qquad \text{for all}\ T>0, \delta >0\]
 and it is a kernel of uniform exponential  $\e$-type $1/2$, i.e. it satisfies the estimate 
 \begin{equation}\label{e:phi}|\Phi_\e((x, t), (z,\tau))|\le c(T)(t-\tau)^{-\frac{3}{4}}
  \Gamma_{z_1,\e}((x, ct), (z,c\tau))
 \qquad 0<t-\tau\le T,
 \end{equation}
with constants independent of $\e$ and of $z_1$. 

As we mentioned in Section \ref{paramethod} the convergence of this 
series relies on properties of convolutions of kernels. 
Hence we will need the following property of the operator $E_{R_{1}, \e}$, 
that ensures that the series can be estimated by a power 
series, so that it is convergent on the mentioned set:
\begin{lemma}\label{lemmaRj} Let  $M_0\subset\Gi$ be a non-characteristic plane. 
For  $z\in M_0$,  $x \in\Gi$,  
with  $|x_1|\leq \e^{2\kappa}$ and for $j\in\N$ it  holds that   
\begin{equation}\label{stimaRj2}|(E_{R_{1},\e})^j R_1((x, t), (z,\tau))| \leq \frac{C^jb_j}{2}
\frac{\Gamma_{z_1, \e}((x, c_1 t), (z,c_1 \tau))}{(t-\tau)^{3/4- j/4}},  
\end{equation} 
$j\in\mathbb{N}$, where $b_j=\Gamma^{j+1}(\frac{1}{4})/\Gamma(\frac{j+1}{4})$ and $\Gamma$ is the Euler 
Gamma-function. 
\end{lemma}
\begin{proof} We argue as in \cite{JSC} or \cite{blu_a} and we prove by induction that 
$$
|(E_{R_{1},\e})^j R_1((x, t), (z,\tau))| \leq \frac{C^jb_j}{2}
\frac{\Gamma_{\e}((x, c_2 t), (z, c_2\tau))}{(t-\tau)^{3/4- j/4}}, 
$$
Here we only sketch the proof, in order to show that the constant is independent of $\e$. The estimate is true for $j=0$ (see \eqref{claim1R1}).  
Let us assume that an analogous estimate holds for $j-1\in \mathbb{N}$. We have 
\begin{align*}
&|(E_{R_{1},\e})^{j} R_1((x, t), (z,\tau))|\le \, 
\\&  \leq\frac{C^{j}b_{j-1}}{2}
\int_{\tau}^{t}(t-\theta)^{-\frac34}(\theta-\tau)^{-\frac34+\frac{j-1}{4}}
\int_{\mathbb{R}^{n}} \Gamma_{\e}((x, c_2 t), (y,c_2 \theta))
  \Gamma_{\e}((y, c_2 \theta),  (z,c_2 \tau))\, dy d\theta.
\end{align*}
By the reproducing property of the fundamental solution, we have
\[\int_{\mathbb{R}^{n}}  \Gamma_{\e}((x, c_2 t), (y,c_2 \theta))
  \Gamma_{\e}((y, c_2 \theta),  (z, c_2 \tau))\, dy = 
\Gamma_{\e}(x, c_2 t), (z,c_2\tau)) 
\]
and, by the change of variable $r=(t-\tau)^{-1}(\theta-\tau)$,
\begin{equation}\label{c.var.}\begin{split}b_{j-1}&\int_{\tau}^{t}(t-\theta)^{-\frac34}(\theta-\tau)^{-1+\frac{j}{4}}\,d\theta
\\=&b_{j-1}
(t-\tau)^{-\frac{3}{4}+\frac{j}{4}}\int_{0}^{1}(1-r)^{-\frac34}r^{-1+\frac{j}{4}}\,dr\,
=b_{j-1}(t-\tau)^{-3/4+{j}/{4}}2^{j/4-3/4}\frac{\Gamma(\frac{1}{4})\cdot \Gamma(\frac{j}{4})}{\Gamma(\frac{j+1}{4})}.
\end{split}
\end{equation} Recall now 
the definition of $b_{j-1}$ and obtain $b_{j}$ from 
 the last integral.
Thus, \eqref{stimaRj2} follows  by induction for all $j\in \mathbb{N}$. 
\end{proof}

\begin{remark}\label{decay}
From the assertion above it follows that the 
convolution of a family of kernels of  uniform exponential $\e$-type $1/2$ with 
a family of kernels with  uniform exponential $\e$-type $\beta$ 
is a family of kernels with  uniform exponential $\e$-type $\beta + 1/2$. 
\end{remark}

For every $\e>0$ the operators $L_{\e}$ and $L_{z_1,\e}$ are uniformly elliptic, so that 
the proof of the convergence of the parametrix method is a well known fact 
(see \cite{Fri}). 
In particular 
$\Gamma_{z_1,\e}$ 
provides a good approximation of $\Gamma_{\e} $ in a neighborhood of the plane. More precisely
 \begin{equation} \begin{split}\label{RdefineJ}
\Gamma_{\e}((x, t), (z,\tau)) &=  \Gamma_{z_1, \e}((x, t), (z,\tau)) 
\\&+
\int_{\mathbb{R}^{n}\times [\tau,t]}  \Gamma_{y_1,\e}((x, t),  (y, \theta))\,
\Phi_\e((y, \theta), (z,\tau))\rho_{\e}(y_1-z_1)  \,  dy d\theta.
\end{split}\end{equation} 
In addition, for $i= 1,\cdots, n$, 
\begin{align}\nonumber &X_{i,0,\e}^2\Gamma_{\e}((x, t), (z,\tau))= X_{i,0,\e}^2\Gamma_{z_1, \e}((x, t), (z,\tau)) 
\\ &+ \quad
\lim_{\delta\to 0^+}\int_{\mathbb{R}^{n}\times [\tau,t-\delta]}X_{i,0,\e}^2\Gamma_{y_1,\e}
((x, t),(y,\theta))\Phi_{\e}((y,\theta),(z, \tau))\rho_{\e}(y_1-z_1)\,dy d\theta.
\label{1sett}\end{align}
Using the explicit representation formulas above we can provide the following estimates for $\Gamma_{\e} - \Gamma_{z_1,\e}$ 
uniform in $\e$: 
\begin{proposition}\label{lemmaJe2}
Let $M_0$ be a non characteristic plane 
 and 
$x=(x_1,\x), z=(z_1,\z) \in \Gi$ such that  $|x_1|, |z_1|\leq \e^{2\kappa}$.  For every $T>0$ there exists a constant $C=C(T)$ such that 
for every   $\e>0$ and for every $t, \tau$ with 
$0<t-\tau\le T$ the following inequalities hold 
 \begin{equation} \label{Rdefine1}
|\Gamma_{\e}((x, t), (z,\tau))  - \Gamma_{z_1,\e}((x, t), (z,\tau))|\le C(t-\tau)^{1/4}\Gamma_{z_1,\e}((x, t), (z,\tau)).
\end{equation}
In addition $\Gamma_{\e}((x, t), (z,\tau))  - \Gamma_{z_1,\e}((x, t), (z,\tau))
$ is a family of kernels of uniform exponential $\e$-type $1/4$
with respect to the vector fields $(X_{i,\e})_{i=2,\cdots, n}.$
\end{proposition}
For the proof we refer to \cite{JSC}, while the uniformity with respect to $\e$ follows by \eqref{uniform heat kernel estimates-e}. 
\begin{proof}[Proof of Theorem \ref{lemmaJe1} ]
We first prove a Riemannian version of Theorem \ref{lemmaJe1}. Precisely we show that 
for all  $(0, \z)$, $(0, \x)$ in $M_0$ and for every $t,\tau$, with $0<t-\tau\le T$  we have 
\begin{equation}\label{zanzara}
	|\hat \Gamma_{0,\e}((\x,t), (\z, \tau)) -\sqrt{4\pi(t-\tau)}\Gn_{\e}((0, \x, t), (0, \z, \tau))|\le   
 C \Gn_\e((0, \x, t), (0, \z,\tau))(t-\tau)^{3/4},
\end{equation}
where $C$ is a constant independent of $\e$.
Indeed, 
$$\sqrt{4\pi(t-\tau)}  \Gamma_{\e} ((0,\x, t), (0,\z, \tau)) - 
\hat \Gamma_{0, \e} ((\x, t), (\z, \tau))= $$
$$=\sqrt{4\pi(t-\tau)}  \Gamma_{0,\e} ((0,\x, t), (0,\z, \tau)) - 
\hat \Gamma_{0, \e} ((\x, t), (\z, \tau)) $$
$$+\sqrt{4\pi(t-\tau)} \Big(\Gamma_{\e} ((0,\x, t), (0,\z, \tau)) - \Gamma_{0,\e} ((0,\x, t), (0,\z, \tau)) \Big). $$
The first difference at the right hand side is zero by Remark \ref{r:splittingNC}, since 
$$\sqrt{4\pi(t-\tau)} \Gamma_{\perp, 0,\e} ((0,\x, t), (0,\z, \tau))={1}$$ and  hence the estimate \eqref{zanzara} follows by  \eqref{Rdefine1}. Always from Proposition \ref{lemmaJe2}
it follows that $\sqrt{4\pi(t-\tau)}  \Gamma_{\e} ((0,\x, t), (0,\z, \tau)) - 
\hat \Gamma_{0, \e} ((\x, t), (\z, \tau))$ is 
a family of kernels of uniform exponential $\e$-type $3/4$
with respect to the vector fields $(X_{i,\e})_{i=2,\cdots, n}.$
Sending $\e$ to $0$ in the assertion \eqref{zanzara}, we obtain that for all  $z=(0, \z)$, $x=(0, \x)$ in $M_0$ and for every $t$ and $\tau$, with $0<t-\tau<T$, we have 
 \begin{equation*}|\hat \Gamma((\x,t), (\z, \tau)) -\sqrt{t-\tau}\Gn((x, t), (z, \tau))|\le   
 C \Gn((x, t), (z,\tau))(t-\tau)^{3/4},\end{equation*}
and the left hand side is a kernel of exponential type $3/4$
with respect to the vector fields   $(X_{i})_{i=2,\cdots, m}.$
Using the Gaussian estimate \eqref{gauss-sol2} of $\Gn$ and $\hat \Gamma$ 
together with formula \eqref{equivdhat} we obtain
$$|\hat \Gamma((\x,t), (\z, \tau)) -\sqrt{t-\tau}\Gn((x, t), (z, \tau))|\le   
 C \hat  \Gn((\x, t), (\hat z,\tau))(t-\tau)^{1/4} $$
and the left hand side is a kernel of exponential type $1/4$
with respect to the vector fields $(\hat X_{i})_{i=2,\cdots, m}.$
Theorem \ref{lemmaJe1} follows immediately.
\end{proof}

\subsection{The reproducing formula for homogeneous sub-Laplacians on a plane}
Here we establish the analogous of Theorem \ref{teorema1}
for homogeneous vector fields 
expressed as in \eqref{struttura campi}, 
under the assumption that the boundary of $D$ 
is the plane $\{x_1 =0\}$. 
This is done integrating in time the result of Theorem \ref{lemmaJe1}.  
Let us first deduce an integral version of Theorem \ref{lemmaJe1}, based on the reproducing formula of the 
heat kernel. 
\begin{lemma}\label{gammaMconvolve}
Let $D=\{(x_1,\hat x)\in \mathbb{R}^{n}:\,x_1>0\}$, and assume that its boundary is non characteristic. 
 There exists $C>0$ such that  for any $(0,\hat x), (0,\hat y)\in \partial D$ and for 
 all $t, \tau$, with $0\leq \tau \le t$ we have 
\begin{equation}\label{eq:heatgammaconvolution}
\hat \Gamma((\hat x, t), (\hat y, \tau))  =
 \end{equation}
$$= \int_\tau^t
\int_{\mathbb{R}^{n-1}}  \Gamma((0,\hat x,t), (0,\hat z, \theta))   
\Gamma((0,\hat z,\theta), (0,\hat y, \tau)) d{\hat z} d \theta +  \hat  R(\hat x, \hat y, t- \tau), 
$$
where
\begin{equation}|\hat R(\hat x, \hat y, t)| \leq Ct^{1/4} \hat \Gamma((\hat x, t), (\hat y,0)),
\label{stimahatR3.13}\end{equation}
and $\hat R$ is a kernel of exponential type $5/2$ with respect to the 
vector fields $\{\hat X_i\}_{i=2,\cdots, n}$.
\end{lemma}
 \begin{proof}
 Let us first prove \eqref{eq:heatgammaconvolution}. To this end we  note that the thesis is true for $t-\tau\geq 1$. Indeed 
$$\hat \Gamma((\hat x, t), (\hat y, \tau))  \leq c
(t-\tau)^{1/4} \hat \Gamma((\hat x, t), (\hat y, \tau))$$
and 
by the standard Gaussian estimates \eqref{gauss-sol2} of the fundamental solution 
and by the relation \eqref{equivdhat} between the distances $\hat d$ and $d$, we obtain
$$\int_\tau^t
\int_{\mathbb{R}^{n-1}}  \Gamma((0,\hat x,t), (0,\hat z, \theta))   
\Gamma((0,\hat z,\theta), (0,\hat y, \tau)) d{\hat z} d \theta $$$$\leq c
(t-\tau)^{3/4} \Gamma((0,\hat x, t), (0, \hat y,\tau))\leq c
(t-\tau)^{1/4} \hat \Gamma((\hat x, t), (\hat y, \tau)).$$
If $t-\tau<1$, by Theorem \ref{lemmaJe1}, we have  
\begin{equation}
\label{e:inventa}  
\Gamma((0,\hat x,t), (0, \hat z, \theta)) = 
\frac{\hat \Gamma((\hat x,t), (\hat z, \theta)) }{\sqrt{t-\theta}}   
(1 + O(t-\theta)^{1/4}).\end{equation}
Thus,
\begin{align*}&\int_{\tau}^t
\int_{\mathbb{R}^{n-1}}  \Gamma((0,\hat x,t), (0,\hat z, \theta))   
\Gamma((0,\hat z,\theta), (0,\hat y, \tau)) d{\hat z} d \theta
\\ &=
\int_{\tau}^t\Bigg(
\int_{\mathbb{R}^{n-1}}  \hat \Gamma((\x,t), (\z, \theta)) \hat\Gamma((\z,\theta), (\y, \tau)) d{\hat z}
\\ &
\quad\quad\quad\left(\frac{1}{(t-\theta)^{1/2}} + O\left(\frac{1}{(t-\theta)^{1/4}}\right)\right)  
\left(\frac{1}{(\theta- \tau)^{1/2}} + O\left(\frac{1}{(\theta - \tau)^{1/4}}\right)\right) \Bigg) d \theta
\\
& =\hat \Gamma((\x,t), (\y, \tau))\int_{\tau}^t\left(\frac{1}{(t-\theta)^{1/2}} +O\left(\frac{1}{(t-\theta)^{1/4}}\right)\right) 
\left(\frac{1}{(\theta- \tau)^{1/2}} +O\left(\frac{1}{(\theta - \tau)^{1/4}}\right)\right)d \theta,\end{align*}
by the reproducing formula. Now, with 
the change of variable $r=(t-\tau)^{-1}(\theta-\tau)$, we get
$$\int_{\tau}^t\left(\frac{1}{(t-\theta)^{1/2}} +O\left(\frac{1}{(t-\theta)^{1/4}}\right)\right)  \left(\frac{1}{(\theta- \tau)^{1/2}} +
O\left(\frac{1}{(\theta - \tau)^{1/4}}\right)\right)d \theta=1 + O\left((t-\tau)^{1/4}\right).$$
Therefore,  we get
\begin{equation}
\label{resto_cappuccio}
\int_{\tau}^t
\int_{\mathbb{R}^{n-1}}  \Gamma((0,\hat x,t), (0,\hat z, \theta))   
\Gamma((0,\hat z,\theta), (0,\hat y, \tau)) d{\hat z} d \theta
\end{equation}
$$
=\hat \Gamma((\x,t), (\y, \tau)) (1 + O((t-\tau)^{1/4}), 
 $$
so that $R$ satisfies \eqref{eq:heatgammaconvolution}.
A similar argument applied to all derivatives
ensures that $\hat R$ is a kernel of exponential type $5/2$ with respect to the 
vector fields $\{\hat X_i\}_{i=2,\cdots, n}$ 
and concludes the proof.
\end{proof}

\subsection{Proof of Theorem \ref{teorema1} for homogeneous vector fields on a plane}

We will provide in Lemma \ref{l:Gammaconvolve} below the proof of Theorem \ref{teorema1} on a plane and for homogeneous vector fields. 
This result can be considered the time independent version of Lemma \ref{gammaMconvolve}.  
It  will be established integrating in time the thesis of that Lemma and 
using the well known fact that the fundamental solutions
$\hat \Gamma_{\hat \Delta}$ of the Laplace type operator \eqref{heat_plane_noe}
and $\Gamma_{\Delta}$ of the Laplace operator \eqref{laplacoperator}
satisfy respectively
\begin{equation}
 \label{hatG-eps}
\hat \Gamma_{\hat \Delta}(\x,\z)=\int_{0}^{+\infty} \hat \Gamma((\x,t)(\z,0)) dt,
\quad \Gamma_{\Delta}(x, z)=\int_{0}^{+\infty} \Gamma((x, t), (z,  0)) dt. \end{equation}

\begin{lemma} \label{l:Gammaconvolve}Let the vector $(X_i)$ be represented as in \eqref{struttura campi}, let 
$M_0=\{(0,\hat x): \hat x\in \mathbb{R}^{n-1}\}$ be a non 
characteristic plane. For any $(0,\hat x), (0,\hat y)\in M_0$ 
\begin{equation}\label{GammaGamma}
{\hat \Gamma}_{\hat \Delta}(\x, \y)  =\int_{\R^{n-1}}  \Gamma_{\Delta}((0,\x), (0,\z))
\Gamma_{\Delta}((0,\z), (0,\y)) d\z + \hat R_{\hat\Delta}(\x, \y),\end{equation}
where
\begin{equation}
	\label{resto_cappuccio_laplaciano}
	 \hat R_{\hat\Delta}(\x, \y)=O(\hat d(\x,\y)^{\frac12}\hat\Gamma_{\hat\Delta}(\x,\y)).\
\end{equation}
In particular $
 \hat R_{\hat \Delta}(\x, \y)$
is a 
kernel of type
$5/2$ in the sense of Definition \ref{kerneltype} with respect to the distance 
$\hat d$ defined on the plane.
\end{lemma}
\begin{proof}
Using \eqref{hatG-eps} and 
integrating both sides of expression (\ref{eq:heatgammaconvolution}) we obtain: 
$$\hat \Gamma_{\hat\Delta}(\x, \y) +  \int_{0}^{+\infty} \hat R((\x,t), (\y, 0)) dt$$
$$
=  
\int_{0}^{+\infty} \int_{0}^{t}\int_{\mathbb{R}^{n-1}} 
\Gamma((0,\x,t -\theta), (0,\z, 0)) \Gamma((0,\z,\theta), (0,\y, 0)) d\z\, d\theta   \,dt.$$
Changing the order of integration,  we get that the last term is equal to 
$$ \int_{\mathbb{R}^{n-1}}
\int_{0}^{+\infty} \left(\int_{\theta}^{+\infty}
\Gamma((0,\x,t -\theta), (0,\z, 0))\,dt\right) \Gamma((0,\z,\theta), (0,\y, 0))  \, d\theta d\z$$ 
$$=\int_{\mathbb{R}^{n-1}}
  \int_{0}^{+\infty}
\Gamma_{\Delta}((0,\x), (0,\z)) \Gamma((0,\z,\theta), (0,\y, 0)) d\theta d\z$$ 
and integrating with respect to $\theta$ we obtain  
$$\int_{\mathbb{R}^{n-1}}\Gamma_{\Delta}((0,\x), (0,\z)) \Gamma_{\Delta}((0,\z), (0,\y))  d\z. $$ 
The estimate of $\hat R_{\hat\Delta}(\x,\y)$ directly  follows easily integrating in \eqref{resto_cappuccio}
and using the estimate of $\hat \Gamma((\x, ct),  (\y,0))$ provided in Lemma \ref{gammaMconvolve}. Therefore, recalling that $\widehat Q=Q-1$ denotes the homogeneous dimension of the plane, we have 
\begin{equation*}
\begin{split}
\hat R_{\hat \Delta}(\x,\y)=& \int_{0}^{+\infty} \hat R((\x,t), (\y,0)) dt \le \\ 
&  \le c \int_{0}^{+\infty}\hat \Gamma((\x, {\tilde c}\, t),  (\y,0))t^{1/4}dt
\le  c\int_{0}^{+\infty} \frac{e^{-\frac{\hat d(\x,\y)^2}{C t}}} {t^{\frac{{\widehat Q}}{2}-\frac14}}\,dt,
\end{split}
\end{equation*}
where the constants may vary from line to line.
Now, with the change of variables $v=-\frac{\hat d(\x,\y)^2}{C t}$ we get
$$\hat R_{\hat \Delta}(\x,\y)
\leq c \int_{0}^{+\infty} e^{-v} \frac{v^{\frac{{\hat Q}}{2}-\frac94}}{{\hat d(\x,\y)}^{{{\widehat Q}}-\frac52}}\,dv \le  c \hat d(\x,\y)^{-{\widehat Q}+2+\frac12}\le c \hat d(\x,\y)^{\frac12}\hat\Gamma_{\hat\Delta}(\x,\y)\,.
$$
 An analogous inequality holds for any derivative
and the result is proved.
\end{proof}

\section{Reproducing formula on a smooth hypersurface} 
\subsection{Reduction of a general hypersurface to a plane with a subriemannian structure} \label{nuova}
Let us denote by $D$ a smooth, open bounded set in $\Gi$ and let 
$0\in \partial D$ be a non characteristic point. In this section we show that 
we can always reduce  the boundary $\partial D$ to the plane $\{(x_1,\hat x): \ x_1=0\}$, via a change of variables. 
Indeed, there exists a neighborhood $V_0$ of $0$ such that 
the subriemannian normal $\nu$ satisfies $$\nu(s) \ne 0 \ \text{for every } s\in \partial D\cap V_0.$$
We can also choose an invariant basis 
$(Z_{i})_{i=1, \cdots, n}$ of the tangent space of $\Gi$ around the point $0$ and 
$Z_{i}$ coincides with the standard element $\partial_i$ of the tangent basis   at the point $0$, for every $i=1,\cdots, n$. In addition, we assume that 
${Z_{1}}_{|0}:={\partial_1}_{|0} = \nu(0)$ and that the vector fields $({Z_{i}}_{|0})_{i=2, \cdots, m}$ span the horizontal tangent space of $\partial D$ at $0$.
We  also assume that the problem is expressed in canonical coordinates of second type
around the point $0$ associated to these vector fields. 
In these coordinates the vector fields admit the representation  
$$Z_1 = \partial_1, \quad Z_i = \partial_i + \sum_{deg(j) > deg(i)} a_{i,j}(s)\partial_j, \text{ for } i=1, \cdots, m$$
while the boundary of $D$ can be identified in a neighborhood $V\subset\subset V_0$ with the graph of a regular 
function $w$, defined on a neighborhood $\hat V=V\cap \R^{n-1}$ of $0$: 
$$\partial D \cap V= \{(w(\hat s), \hat s): \hat s\in \hat V  \}.$$
By the choice of coordinates we have in particular that 
\begin{equation}\label{zetaiszero}
Z_{i}w(0) =0.
\end{equation}
On the set $V$ the function 
$\Xi(s_1, \hat s) = (s_1 - w(\hat s), \hat s) $
is a diffeomorphism. It sends  $\partial D\cap V$ to a subset of the plane $\{x_1 =0\}$:
$$\Xi(\partial D \cap V) =\{(x_1, \hat x): x_1 =0\}.$$
Through this change of variables
the vector fields $Z_i $ can be represented as 
\begin{equation}\label{campinonomo}
\begin{array}{ll}X_1 & = d\Xi(Z_1) = \partial_{x_1}, \\ 
X_i &= d\Xi(Z_i) = \partial_{x_i} + \sum_{deg(j) > deg(i)} a_{i,j}(x_1 + w(\hat x), \hat x)\partial_{x_j} + Z_i w (\x)\partial_{x_1},
\end{array}
\end{equation}
for $i=1,\cdots, n$, where the polynomials $a_{ij}$ are the same of the ones defined in \eqref{struttura campi}. 

A neighborhood of $0$ in the boundary of $D$ locally becomes in the new coordinates an open subset  $M_0= \Xi(\partial D \cap V)$ of the plane $\{x_1 =0\}$. We can restrict the vector fields 
$(X_i)_{i=1,\cdots,m}$ to the tangent to $M_0$ 
and we call them $\hat X_{i}$:
\begin{equation} \label{hatticampi}\hat X_{i} = \partial_{i} + \sum_{deg(j) = deg(i) + 1}^\kappa 
a_{i,j}(w(\hat x), \hat x)\partial_{j}, \quad i=2,\cdots,n.\end{equation}

The vector fields $(\hat X_{i})_{i=1, \cdots, m}$ still satisfy the assumption 
\eqref{assumption}, which ensures that they satisfy the 
H\"ormander finite rank condition \cite{hormander}. 
They do not define a general H\"ormander structure (see \cite{Montgomery}), 
since they have been obtained from the generators of a Carnot group via a change 
of variables. It is important to note that the vector fields $X_i$ as well as the vector fields $\hat X_i$ are not 
homogeneous with respect to the new variables $x_i$. However we will see in Lemma \ref{teoRS} and Lemma \ref{teo2RS} that at every point they admit approximating vector fields respectively 
$Z_i$ and $\hat Z_i$ homogeneous in the new variables.
Hence the local homogeneous dimension of $\R^{n}$ endowed with the choice of the vector fields $X_i$ is $Q$. 
Since the H\"ormander condition is satisfied, a Carnot Carath\'eodory distance $d$ is defined 
in terms of the vector fields $(X_i)_{i=1}^m$. 
Thanks to assumption \eqref{assumption}, the vector fields $\{\hat X_{i}\}_{i=2, \cdots, m}$, 
defined in \eqref{hatticampi}, generate on the plane  $M_0$ 
a subriemannian structure with local homogeneous dimension $\hat Q= Q-1$
and induce a distance $\hat d$ on  $M_0$ defined through the exponential map 
as in \eqref{2.5bis}, 
which satisfies \eqref{finalmente} and 
\eqref{equivdhat}. 

The Laplace type operator, analogous to \eqref{laplacoperator} 
and expressed in terms of the vector fields $X_i$ is denoted by
\begin{equation}\label{ohscusa}\Delta = \sum_{i=1}^m X_i^2 + \sum_{i=1}^m b_i X_i\end{equation} 
and it has a fundamental solution $\Gamma_\Delta$, 
of class $C^\infty$ out of the pole (see for example \cite{RS}). The operator 
\eqref{heat_plane_noe}, expressed in terms of the vector fields 
$\hat X_i$ becomes 
\begin{equation}\label{scusa}\hat \Delta = \sum_{i=2}^m {\hat X}_i^2,
\end{equation}
with fundamental solution 
$\hat \Gamma_{\hat \Delta}$. In analogy with the definition of type of a kernel with respect to the 
vector fields $(X_i)$, given in \eqref{e:sileva}, 
 we  give here the definition of kernel of local  type
$\lambda$ with respect to the vector fields $\hat X_2,\cdots,\hat X_n$:
\begin{definition}\label{defikernel}
  $k$ is a kernel of local type
$\lambda$ with respect to the vector fields $\hat X_2,\cdots,\hat X_n$
and the distance $\hat d$ if it is a smooth function out of the 
diagonal and, in   any open set
$V$, the following holds:
 for every $p$ there exists a positive constant
$C_p$
such that, for every
$\hat x,\hat y
\in \partial D\cap V$,   $\hat x \not=\hat y$,  
$$|\hat X_{i_1},\cdots,\hat X_{i_p} k(\hat x,\hat y)|\leq C_p\,\hat d(\hat x,\hat y)^{\lambda - p-2}
\frac{\hat d(\hat x,\hat y)^{2}}{|B(\hat x, \hat d(\hat x, \hat y))|}.$$
\end{definition}
Clearly, if the space is homogeneous, the previous definition coincides with Definition \ref{kerneltype}.

\subsection{A freezing procedure}
Here we will show that, when we are studying pointwise properties around 
a fixed point $x_0$, we can always reduce our vector fields to homogeneous ones.  
The proof is made approximating the vector fields with nilpotent ones, adapting to this 
context the Rothschild and Stein parametrix method. 
In the classical case the vector fields are lifted to vector fields free up to step $\kappa$ 
and then they are reduced to the generators of a free algebra with a freezing method. 
Here we can not lift the vector fields to free ones otherwise we would loose assumption \eqref{assumption}. 
However, we can 
use the explicit expression of the vector fields \eqref{campinonomo} to obtain an ad hoc version of the Rothschild and Stein method. 

Let $D$ be a smooth, open bounded set in $\R^n$ 
As shown in the previous section we can 
assume, up to a change of variable, that 
$0\in \partial D$ and that there exists $V$ such that 
$$\partial D \cap V =\{(x_1, \hat x): x_1 =0\}.$$
Moreover, there exists a regular function $w$ such that 
the vector fields $(X_i)$ can be represented as 
\eqref{campinonomo}. We prove the following result analogous to 
 \cite{NSW} in our simplified setting:
\begin{proposition} \label{proteta}
There exist  open neighborhoods
$U$ of $\;0$ in $\R^n$ and $V,\, W$ of $\;0\in \partial D \subset \R^n$, with $W \subset V$ and, for every $z$ fixed  in $W$, 
a change of coordinates $\Xi_z$ such that 
\begin{itemize}
\item {the function $x\rightarrow \Xi_z(x)$
is a diffeomorphism from $U$ on the image}
\item {in the new coordinates the vector fields will admit the following representation:
$$ \Xi_z (X_{1}) = \partial_{y_1}$$$$ d \Xi_z (X_{i})=\partial_{y_i} + \sum_{deg(j)>deg(i)} a_{i, j} (y_1 + w_z(\y), \y)\partial_{y_j} +
X_i w_z \partial_{y_1},\quad i=2,\cdots,n. $$
}
\end{itemize}
\end{proposition}
\begin{proof}
 Let us call $M= \{(w(\x), \x):
(0,\x) \in V\cap \partial D \},$ where $w$ is the function introduced above and defining the vector fields in \eqref{campinonomo}. 
For every $z\in M$ we will denote $T_z$ the group homomorphism which sends $0$ to $z$ 
and whose differential 
sends ${X_1}_{|0}$ in the normal $\nu(z)$ at $M$ in $z$ and 
 ${X_2}_{|0}, \cdots, {X_n}_{|0}$ in a basis 
of the tangent space to $M$ at $z$.
If we fix $z$ the implicit function theorem, (see \cite{FSSC1}, \cite{cittiman06}) ensures that there exists 
a neighborhood $U = I \times \hat U$ of $0$ and
a function $w_z: \hat U \rightarrow  \mathbb{R}$ such that $w_z(0)=0$ and
$$\{(w_z (\y), \y): \y\in  \hat U\}= T_z^{-1}(M)\cap U,$$
so that $\{T_z(w_z (\y), \y): \y\in  \hat U\}\subset M\cap V$.  We can always assume that  $\nabla_z w_z(0)=0$.  
Due to the regularity of the boundary we can find an open set $W\subset V$ 
such that for every $z\in W\cap M$ the function $w_z$ is defined on the same set $\hat U$
with values in the same set $I$. Hence we can define the map
$$E_z : U \rightarrow V, \quad E_z(y_1, \y):= T_z(y_1 + w_z (\y), \y). $$
$ E_z$ is invertible on its image and sends the plane $\{y_1=0\} \cap U$ into a suitable subset 
of $M$. The composition 
$E_0^{-1} E_z$ sends the plane $\{y_1=0\}$ into the 
the plane $\{x_1 = 0\}$, boundary of $D$. 
For every $z\in W\cap M$ its inverse function 
$\Xi_z(x)$
is a diffeomorphism on the image and $ \Xi_z(W) \subset U\subset \Xi_z(V)$ 
The vector fields $X_i$ can be represented as follows in the new coordinates (see also \cite{ASV}, \cite{CM}):  
\begin{equation*}
\begin{split}
d \Xi_z (X_{1}) & = \partial_{y_1}\\
d \Xi_z (X_{i}) & =\partial_{y_i} + \sum_{deg(j)>deg(i)} a_{i, j} (y_1 + w_z(\y), \y)\partial_{y_j} +  X_i w_z(\y) \partial_{y_1}, \quad i=2,\cdots ,n.
\end{split}
\end{equation*}
\end{proof}
 
We can now prove the following result, analogous to Theorem 5 in \cite{RS}:
\begin{lemma} \label{teoRS}
With the same notations of the previous Proposition \ref{proteta}, 
let us call
$$Z_{i}=  \partial_{y_i}+ \sum_{deg(j)>deg(i)}  a_{ij}(y) \partial_{y_j},$$
for $i=1, \cdots, n$. Then we have
$$d \Xi_z(X_{i }) - Z_{i} = R_{i, z, \Xi}$$
where $R_{i, z, \Xi}$ are vector fields of local degree $\leq deg(i)-1$ depending smoothly on $z$. Precisely
$R_{i, z, \Xi}=  \sum_j r_{ij, z}\partial_j$
where $r_{ij, z}=O(d(x,y)^{deg(j)-deg(i)+1})$. 
\end{lemma} 
\begin{proof}
It is a direct computation. Indeed the assertion is true for $i=1$. For every $i>1$ the difference 
$d \Xi_z(X_{i }) - Z_{i}$ 
can be expressed as 
$$d \Xi_z(X_{i }) - Z_{i} = \sum_{deg(j)>deg(i)}  \Big(a_{ij}(y_1 + w_z(\hat y), \hat y) - a_{ij}(y_1, \hat y)\Big)\partial_{y_j} + X_i w_z(\hat y)\partial_{y_1}.$$
We first note that, since  $w_z(0)=0$ and we can always think that also  $X_i w_z(0)=0$, then 
$X_i w_z(\hat y)\partial_{y_1}$ is an operator of degree 0. 
Moreover, being $a_{ij}$ homogeneous polynomials, 
 their difference can be represented as a homogeneous polynomial. Precisely 
there exists a suitable polynomial $a^1_{ij}$ homogeneous of degree $deg(i)-deg(j)-1$ such that 
$$a_{ij}(y_1 + w_z(\hat y), \hat y) - a_{ij}(y_1, \hat y)= w_z(\hat y)  
a^1_{ij}(y_1,y_1 + w_z(\hat y), \hat y) =$$$$= O(\|\hat y\|^2)
a^1_{ij}(y_1,y_1 + w_z(\hat y), \hat y),$$
since $w_z$ and its gradient vanish at $\hat y=0$. 
\end{proof}

A similar relation holds between the vector fields restricted to the boundary: 
\begin{lemma} \label{teo2RS} Using the same notation of Proposition \ref{proteta} and 
setting $$\hat Z_{i}=  \partial_{y_i}+ \sum_{deg(j)>deg(i)}  a_{ij}(0, \hat y) \partial_{y_j},$$ we get: 
$$d {\hat \Xi}_z(\hat X_{i }) = {\hat Z}_{i} + {\hat R}_{i, z, \Xi}$$
where ${\hat R}_{i, z, \Xi}$ are vector fields of local degree $\leq deg(i)-1$ depending smoothly on $z\in W$. 
\end{lemma} 
\begin{proof} We omit the proof which is exactly the same as the 
previous lemma. 
\end{proof}

\subsection{Properties of the fundamental solution and its approximating ones}
The vector fields 
$(X_i)_{i=1, \cdots, n }$ in \eqref{campinonomo}, as well as their restriction to 
the boundary $(\hat X_i)_{i=2, \cdots, n}$, are in general non homogeneous in the variables $x$, 
but we have proved in the previous section that for every $z$ their images through $\Xi_z$ admit 
homogeneous approximating vectors fields. Then calling 
$X_{i, z } = d {\Xi_z}^{-1}({Z}_{i})$ 
for $i=1, \cdots, n,$ and applying the change of variable $\Xi$ 
to the result of Lemma \ref{teoRS}, we deduce that for every 
$i=1,\cdots, n$ 
there exists an operator $R_{i, z }$ such that $R_{i, z }\le deg(i)-1$ and
\begin{equation}\label{approssimacampo} X_{i } = X_{i, z } + {R}_{i, z}. \end{equation}
Calling $  \hat X_{i, z}= d {\hat \Xi_z}^{-1} (\hat Z_{i})
$ for $ i=2,\cdots, n,$ we obtain from  Lemma \ref{teo2RS} that for every  
$i=2,\cdots, n$ there exists a vector field
$  \hat R_{i, z}$ such that $ \hat R_{i, z}\leq deg(i)-1$ and 
$$ \hat X_{i } =  \hat X_{i, z}+ {\hat R}_{i, z}.$$
The associated sub-Laplacian type operators
are defined as
\begin{equation}\label{ponteggio}
	\Delta_z = \sum_{i=1}^m X_{i, z}^2,\quad \hat \Delta_z = \sum_{i=2}^m {\hat X}_{i, z}^2,\quad 
 \Delta_Z = \sum_{i=1}^m Z_{i}^2,\quad \hat \Delta_Z = \sum_{i=2}^m {\hat Z}_{i}^2,
\end{equation}
with fundamental solutions 
$\Gamma_{z, \Delta}$, $\hat \Gamma_{z, \hat \Delta}$, 
$\Gamma_{\Delta_Z}$ and $\hat \Gamma_{\hat \Delta_Z}$ respectively. 
Note that  $\Gamma_{\Delta_Z}$ and $\hat \Gamma_{\hat  \Delta_Z}$ do not depend on the fixed point $z$. 
 
We can now apply the parametrix method  of 
\cite{JSC}, recalled in \eqref{piantina} and \eqref{e: informal serie} to 
estimate the fundamental solutions 
$\Gamma_\Delta$ and $\hat \Gamma_{\hat \Delta}$, associated to the operators \eqref{ohscusa} and \eqref{scusa} respectively. 
The argument is similar to the one applied 
in Section \ref{s:LaplaceBeltrami} but  in this case 
the proof is standard, since we do not have to take 
care of the different homogeneous dimensions of the Riemannian and 
subriemannian structures. Hence we state without proof the following lemma:
\begin{lemma} \label{lemmagamma} 
Let us consider the operators defined in \eqref{ponteggio}. Then 
$$H=\Delta - \Delta_z \quad \hat H=\hat\Delta - \hat\Delta_z $$
are differential operators of degree 1. As a consequence 
\begin{equation} \nonumber
\Gamma_\Delta(x,z)  -  \Gamma_{z, \Delta}(x, z)=
\Gamma_\Delta(x,z)  - \Gamma_{\Delta_Z}(\Xi_z(x), 0) 
\end{equation}
are kernels of type $3$, with respect to the vector fields 
$X_i$ and the distance $d$. Analogously
\begin{equation} \nonumber
\hat \Gamma_{\hat \Delta}(\hat x,  \hat z)  -  \hat \Gamma_{z, \hat \Delta}(\x, \hat z) 
=\hat \Gamma_{\hat \Delta}(\x, \hat z)  - \hat \Gamma_{\hat \Delta_Z}(\hat \Xi_{\hat  z}(\hat x), 0) 
\end{equation} 
are kernels of type $3$ with respect to the 
vector fields 
$\hat X_i$ and the distance $\hat d$. \end{lemma}

We will also denote  $(X_i)^*$ the formal adjoint of $X_i$. 
\begin{remark}
Let us note that for every $i=1,\cdots, n$, the vector field $X_i$ is no more self adjoint, but its 
formal adjoint differs from $X_i$ by an operator of order 0. Indeed there exists a smooth function 
$\phi_i$ such that 
\begin{equation}\label{aggiuntoXX}(X_{i})^*= -X_{i} + \phi_i, \quad i=1,\cdots,n.\end{equation}
Indeed
$$(X_i)^* =  - X_i - \sum_{deg(j) > deg(i)} \sum_k \partial_{x_1} a_{i,j}(x_1 + w(\hat x), \hat x)\partial_{x_k} w.
$$
\end{remark}

In the sequel we will denote $X_i^z$ the derivative with respect to $z$ and $X_i^x$ the one 
with respect to $x$ of a kernel $K(x,z)$.
From Proposition 5.10 in \cite{CC} (see also \cite{RS}, page 295, line 3 from below)
 we have
\begin{prop}\label{uniformderivatives}
Assume that $f\in C^{\infty}_0(\R^{n-1})$, and for $x\in \R^n$ define 
$$F(f)(x):= \int_{\R^{n-1}} \Gamma_\Delta (x, (0, \hat y)) f(\y) d\y.$$
 For every $i, {\,h}=1, \cdots, m$ 
there exist kernels 
$ \Gamma_{i,h}(x, y)$ and $S_{i}(x, y)$, 
of type $2$ with respect to the distance $d$, such that 
$$X_i F(f)(x) = $$
$$=-\int_{\R^{n-1}} 
\sum_{h=1}^m  (X^{y}_{h})^*  \Gamma_{i,h }(x, (0, \hat y))    f(\y) d\y - 
\int_{\R^{n-1}} S_{i}(x,(0, \hat y))   f(\y) d\y.$$
\end{prop}
\begin{lemma} \label{gammanabla}Let $f\in C_0^\infty(\R^n)$. 
Let us call $$G(f)(x) := \int_{\R^n} \Gamma_\Delta(x,y)f(y)dy$$
Then there exists a kernel $S$ of type 1 such that the 
 operator $G_1(f):= G(\nabla f)$ can be represented as 
$$G(\nabla f) = E_S(f), $$
where $E_S$ is the operator with kernel $S$.
\end{lemma}
\begin{proof} 
We have 
$$G(X_i f) = \int \Gamma_\Delta(x,z) X_i^z f(z) dz = \int (X_i^z)^*\Gamma_\Delta(x,z)  f(z) dz.$$
Hence we only have to prove that the kernel 
$$S:=(X_i^z)^*\Gamma_\Delta(x,z)$$
is a kernel of type 1 with respect to the distance $d$. By \eqref{aggiuntoXX} there exist regular functions $\phi_i$ such that 
$$(X_i^z)^* =  - X_i^z + \phi_i.$$
On the other side, by \eqref{approssimacampo} 
for every $i=1,\cdots, n$ 
there exist and operator $R_{i, z }$  such that $deg(R_{i, z })\le deg(i)-1$ and
$$ X^z_{i } = X^z_{i, z } + R^z_{i, z}. $$
Finally in \cite{RS}, page 295, line 3 from below, it is proved that 
$$X^z_{i, z } \Gamma_{z, \Delta}$$
is an kernel of type 1. Now we use the fact that 
$K=\Gamma_{\Delta}- \Gamma_{z, \Delta}$ is an operator of type 3, 
to conclude that 
$$S= (X^z_i)^* \Gamma_\Delta = (- X^z_{i, z } - R^z_{i, z} + \phi_i)(\Gamma_{z, \Delta} + K)$$
is a kernel of local type 1 with respect to the distance $d_z$ associated with the 
vector fields $X_{i,z}$. On the other side as in Lemma 
\ref{distanze-fuori}, the distances $d$ and $d_z$ are equivalent, so that 
the conclusion follows.
\end{proof}

\subsection{The reproducing formula for non homogeneous vector fields}
In this section we prove Theorem \ref{teorema1}.  
The proof is obtained, via the results of the previous 
section,
by  reducing to the analogous result for homogeneous vector fields, already established in Lemma \ref{l:Gammaconvolve}. 
\begin{proof}[\bf{Proof of Theorem \ref{teorema1}}]
By Lemma \ref{lemmagamma}, 
\begin{equation}\label{lenotazioni}{\hat \Gamma}_{\hat \Delta}(\x, \z)  - 
{\hat \Gamma}_{\hat \Delta_Z}(\hat \Xi_{\hat  z}(\hat x), 0)  \end{equation}
is a kernel of type $3$ with respect to the vector fields $\hat X_i$. 
For the vector fields $(Z_i)_{i=1, \cdots, n}$ and the fundamental solution associated to the 
corresponding sub-Laplacian type operator, we can apply Lemma \ref{l:Gammaconvolve}, so that 
$$ {\hat \Gamma}_{\Delta_Z}(\hat \Xi_{\hat  z}(\hat x), 0) - 
\int_{\R^{n-1}}   \Gamma_{\Delta_Z}((0,\x), (0,\y))
\Gamma_{\Delta_Z}((0,\y), (0,\z)) d\y $$
is a kernel of type $5/2$ with respect to the vector fields $\hat X_{i,z}$. 
Using Lemma \ref{teo2RS} we deduce that a kernel has the same type 
with respect to the vector fields $\hat X_i$ and $\hat X_{i,z}$. 
Inserting in \eqref{lenotazioni} we get that
\begin{equation}\label{labelaggiunta}{\hat \Gamma}_{\hat \Delta}(\x, \z)  - \int_{\R^{n-1}}   \Gamma_{\Delta_Z}((0,\x), (0,\y))
\Gamma_{\Delta_Z}((0,\y), (0,\z)) d\y \end{equation}
is a kernel of type $5/2$.  
Applying again Lemma \ref{lemmagamma} we deduce that 
the following difference, 
$$\int_{\R^{n-1}}  \Gamma_{\Delta_Z}((0,\x), (0,\y))
\Gamma_{\Delta_Z}((0,\y), (0,\z)) d\y - 
\int_{\R^{n-1}}  \Gamma_{\Delta}((0,\x), (0,\y))
\Gamma_{\Delta}((0,\y), (0,\z)) d\y =$$
$$\int_{\R^{n-1}}  \Gamma_{\Delta_Z}((0,\x), (0,\y))
\Big(\Gamma_{\Delta_Z}((0,\y), (0,\z)) d\y - 
\Gamma_{\Delta}((0,\y), (0,\z)) \Big) d\y +$$
$$+  
\int_{\R^{n-1}}  \Big(\Gamma_{\Delta_Z}((0,\x), (0,\y))
 - \Gamma_{\Delta}((0,\x), (0,\y))\Big)
\Gamma_{\Delta}((0,\y), (0,\z)) d\y,$$
is a kernel of type $3$.
As a consequence, we deduce from here and \eqref{labelaggiunta} that 
$${\hat \Gamma}_{\hat \Delta}(\x, \z)- 
\int_{\R^{n-1}}  \Gamma_{\Delta}((0,\x), (0,\y))
\Gamma_{\Delta}((0,\y), (0,\z)) d\y $$
is a kernel of type $5/2$.  The proof is complete.
\end{proof}

\section{Poisson kernel and Schauder estimates at the boundary}

In this section we will show the existence of a Poisson kernel for 
the Dirichlet problem, stated in Theorem 
\ref{mainRn}. From this we deduce the Schauder estimates at the boundary stated in 
Theorem \ref{c:schauderGroups}. 

Consider a bounded smooth set $D$ and a sub-Laplacian type operator $\Delta$ 
defined in $D$, 
as in \eqref{laplacoperator}, in terms of the homogeneous 
vector fields defined in \eqref{struttura campi}. 
The corresponding Dirichlet problem is expressed as
 \begin{equation}
\label{festa}
\Delta u=f \ \text{in }D, \quad u=g  \   \text{on }\partial D,
\end{equation}
for a suitable boundary datum $g$ and a smooth function $f$ defined on $D$.

As mentioned in Section \ref{nuova}, 
we can locally perform a change of variable, and reduce the domain of the Dirichlet problem 
to the half space. Hence there is an open set $V\subset \R^n$ such that $D\cap V= \{x=(x_1,\x)\in V: x_1>0\}$ and $\{x_1 =0\}$  is a non characteristic plane. Under this 
change of variable, the vector fields 
$\{X_i\}_{i=1,\cdots,m}$ will take the non homogeneous expression of 
\eqref{campinonomo}. 
Their restriction to the boundary, denoted $(\hat X_i)$ is defined in \eqref{hatticampi} induces on the set $\partial D$ a distance $\hat d$ defined in \eqref{finalmente}.  
The corresponding spaces of H\"older continuous functions, will be denoted $\hat C^{k, \alpha}$. 

We look for a Poisson operator 
in a neighborhood $V$ of a point $x_0 \in \partial D$. 
 We say that $P:C^{\infty}( V\cap\partial D)\rightarrow C^{\infty}(V\cap \overline{D})$ is a local Poisson operator 
for the   problem \eqref{festa} if,  for every $g\in C^{\infty}( V\cap\partial D)$, the function $u:=P(g)$
satisfies 
$\Delta u=0$ in $D\cap V$ and  $u(x)=g(x)$ for all $x\in \partial D\cap V$. 
We will construct a parametrix for the Poisson kernel of the Dirichlet problem, 
adapting to the present setting a method introduced by Greiner and Stein \cite{GreinerStein} 
and Jerison \cite{Jerison}. They used an approximating kernel, 
defined via pseudodifferential instruments, 
while we use here the kernel found in Theorem
\ref{teorema1}.  We will denote it as follows:
$$\hat\Gamma_{\Delta^2}(\x, \y):= \int_{\R^{n-1}}  \Gamma_{\Delta}((0,\x), (0,\z))
\Gamma_{\Delta}((0,\z), (0,\y)) d\z.$$
 
In analogy with \eqref{erreuno} we call 
$$R_1(\x, \y) := \hat \Delta \Big({\hat \Gamma}_{\hat \Delta}(\x, \y)  - \hat\Gamma_{\Delta^2}(\x, \y)\Big).$$
In the present case  $R_1$ is a kernel of type $1/2$ with respect to the distance $\hat d$. 
As in \eqref{Fi} we now call 
$$\Phi(\hat x, \hat y):= \sum_{j=0}^\infty (E_{R_1})^j(R_1)(\hat x, \hat y). $$
Using a standard singular integral  argument, 
we deduce that the series uniformly converges on any bounded open set $V_0$ and it is a kernel of type $1/2$, 
As a consequence 
\begin{equation}\label{typekernelfi}
\begin{split}
&\int_{\mathbb{R}^{n-1}\cap V_0}  \Gamma_\Delta((0, \hat x),(0, \hat z))\Phi(\hat z, \hat y) d\hat z \;\; \text {is of type 3/2  with respect to the distance }\hat d \\
& E_{\hat\Gamma_{\Delta^2}}(\Phi(\x, \y)) \;\; \text {is  of type  5/2  with respect to the same distance},
\end{split}\end{equation}
where $E_{\hat\Gamma_{\Delta^2}}$ denotes the operator with kernel $\hat\Gamma_{\Delta^2}$.
In addition the fundamental solution of the operator $\hat \Delta$ can be represented as
\begin{equation}\label{semprepara}
\hat \Gamma_{\hat \Delta}(\x, \y) = \hat\Gamma_{\Delta^2}(\x, \y) + E_{\hat\Gamma_{\Delta^2}}\Phi(\x, \y),
\end{equation}
for $\x, \y\in V_0\cap \R^{n-1}.$

Let us now prove Theorem \ref{mainRn} with 
$$R(g)(\hat y) :=\int_{\mathbb{R}^{n-1}\cap V_0} \int_{\mathbb{R}^{n-1}\cap V_0} \Gamma_\Delta((0, \hat y), (0,\hat s)) \Phi((0, \hat s), (0,\hat z)) \hat \Delta g(\hat z) d\hat s d\hat z  $$
and 
\begin{equation}\label{cappa}
K: {\hat C}^{2}(\partial D\cap V_0) \rightarrow \hat C(\partial D\cap V_0), \quad K= K_1 + R. 
\end{equation}

\begin{proof}[{\bf Proof of Theorem \ref{mainRn}}]
Since we are proving a local property, it is not restrictive that the 
boundary datum $g$ belongs to $C^\infty_0 (\partial D \cap V_0)$. 
Since $\Gamma_{\Delta}$ is the fundamental solution of $\Delta$,	
then the function $u=
P(g) (x)$
satisfies 
$\Delta u=0 $ in $ V\cap D.$ 
Hence by  \eqref{semprepara}, we have
$$\nonumber  P(g) (0, \x)=
\int_{\mathbb{R}^{n-1}} \Big(\hat\Gamma_{\Delta^2}(\x, \y) + E_{\hat\Gamma_{\Delta^2}}(\Phi(\x, \y))\Big)\hat  \Delta g(\y) d\y =$$$$=
\int_{\mathbb{R}^{n-1}} \hat \Gamma_{\hat \Delta}(\x, \y))
\hat  \Delta g(\y) d\y   =g(\x).$$
\end{proof}

Once proved the existence of a Poisson kernel, the proof of Schauder 
estimate is based on properties of singular integrals. We follow here the same ideas as in \cite{Jerison} 
and we prove that the operator $P$ is bounded. Since it can be represented as 
in \eqref{poissonintro} we will start with the properties of $K$.

Let us first note that both $K_1$ and $R$ can be extended to 
operators with values in $C( D \cap V)$  setting 
$$K_1(g)(y) =\int_{\partial D \cap V_0} \Gamma_{\Delta}(y, (0,\hat z)) \hat \Delta g(\hat z) 
d\hat z.$$
$$R(g)(y) =\int_{\mathbb{R}^{n-1}\cap V_0} \int_{\mathbb{R}^{n-1}\cap V_0} \Gamma_\Delta(y, (0,\hat s)) \Phi((0, \hat s), (0,\hat z)) \hat \Delta g(\hat z) d\hat s d\hat z.$$
As a consequence $K= K_1 + R $ will be considered as an operator acting between the 
following sets
$$K: \hat C^{2}(\partial D\cap V_0) \rightarrow C(D\cap V_0).$$ 

\begin{remark}\label{finiremomai}
Let us explicitly note that the spaces $C^{k, \alpha}$ associated with the 
vector fields $X_i$ defined in \eqref{campinonomo} 
are equivalent to the spaces $C^{k, \alpha}$ associated with the 
vector fields 
\begin{equation}\label{campinonomodue}
\begin{array}{ll}\Y_1 & = \partial_{x_1}, \\ 
 \Y_i & = \partial_{x_i} + \sum_{deg(j) > deg(i)} a_{i,j}(x_1 + w(\hat x), \hat x)\partial_{x_j}, i=1, \cdots, n
\end{array}
\end{equation}
since these vectors are linear combinations of the previous ones. 
\end{remark}
	\begin{lemma}
Let $D= \{(x_1,\x)\in \R \times \R^{n-1}: x_1>0\}$ be a half space with non characteristic boundary. 
Then for every $V\subset \subset V_0$ there is a constant $C_1$ such that for every $g\in {\hat C}^{2, \alpha}(\partial D \cap V_0)$
\begin{equation}\label{tesicappag}
\|K(g)\|_{ C^{1, \alpha}( D \cap V)}\leq C_1 \|g\|_{{\hat C}^{2, \alpha}(\partial D \cap V_0)}.\end{equation}
In addition there is a constant $C_2$ such that 
if $g\in C_0^\infty(\partial D \cap V_0)$, then 
$$K(g)\in \left\{\phi:\, |\phi(0,\z)|\le C_2  \frac{\hat d(\z,\operatorname{supp}(g))}{|\hat B (\z, \hat d(\z,\operatorname{supp}(g)))|}\  \forall \hat z \ \\ \text{s.t. }
\hat d(\hat z, \operatorname{supp}(g)) \geq 2 
\operatorname{diam}(\operatorname{supp}(g)) \right\}.$$
	\end{lemma}
		\begin{proof} 
Clearly $\Gamma_{\Delta}((0,\z), (0,\y))$ is a kernel of type $2$ 
	with respect to the distance $d$ in the sense of Definition 
	\ref{defikernel}. Because of inequality 
	\eqref{equivdhat} we deduce that there are constants $C_1, C_2$ such that 
\begin{equation}\label{homo} 
C_1  \frac{\hat d(\z,\y)}{|\hat B (\z, d(\z,\y))|} \leq \Gamma_{\Delta}((0,\z), (0,\y))\leq C_2 \frac{\hat d(\z,\y)}{|\hat B (\z, d(\z,\y))|}
	\end{equation}
	so that $\Gamma_{\Delta}((0,\z), (0,\y))$ is a kernel of type $1$ with respect 
	to the distance $\hat d$  induced on $\partial D$, 
while the first derivatives of $\Gamma_{\Delta}((0,\z), (0,\y))$ 
	are 
	singular integrals. As a consequence  we obtain (see for example \cite{NagelStein}) 
\begin{equation}\label{nucleokappa1}
\|E_{\Gamma_\Delta}(\phi)\|_{ C^{1, \alpha}( D \cap V)}	\leq C \|\phi\|_{{\hat C}^{ \alpha}(\partial D \cap V_0)}.
\end{equation}
for every $\phi\in {\hat C}^{\alpha}(\partial D \cap V_0)$, 
where $E_{\Gamma_\Delta}$ denotes the operator with kernel $\Gamma_{\Delta}(z, (0,\y))$.
Therefore $K_1 = E_{\Gamma_{\Delta}} \circ\hat \Delta$ satisfies
$$
\|K_1(g)\|_{ C^{1, \alpha}( D \cap V)}	\leq C \|\hat \Delta g\|_{{\hat C}^{\alpha}(\partial D \cap V_0)}\leq C \|g\|_{{\hat C}^{2, \alpha}(\partial D \cap V_0)}.
$$	
Since  
$\Phi$ is a kernel of type $1/2,$ its associated operator  $E_{\Phi}$  satisfies
\begin{equation}
\|E_{\Phi} (\hat \Delta g)\|_{{\hat C}^{ \alpha + 1/2}(\partial D \cap V_0)}\leq 
 C \|\hat \Delta g\|_{{\hat C}^{\alpha}(\partial D \cap V_0)}\leq C\|g\|_{{\hat C}^{2, \alpha}(\partial D \cap V_0)}.\end{equation}
It follows that
	$$\|R(g)\|_{ C^{1, \alpha}( D \cap V)}	=\|E_{\Gamma_\Delta}E_{\Phi}(\hat \Delta g)\|_ {C^{1, \alpha}( D \cap V)}\leq 
\|E_{\Phi}(\hat \Delta g)\|_ {C^{\alpha}( D \cap V)}\leq \|g\|_{{\hat C}^{2, \alpha}(\partial D \cap V_0)},$$
	
	In particular \eqref{tesicappag} directly follows.  Also the decay property of $K$ immediately follows, since 
 $$d(\z,\y)\ge d(\z,  \operatorname{supp}g)\; $$  
 for all $\y\in \operatorname{supp}g$  and for all $\z$ such that  $\hat d(\z, \text{supp}\,g)\ge 2\operatorname{diam}(\operatorname{supp}g)$.
	\end{proof}

Arguing as in Remark \ref{finiremomai} we have the following
\begin{remark}
There are $C^\infty$ functions such that the Laplace type operator $\Delta$ 
can be expressed as 
$$\Delta = \Y_1^2 + \sum_{i=2}^m(\Y_i - Z_i w \Y_1)^2 + b_1\Y_1 + \sum_{i=2}^m b_i (\Y_i - Z_i w \Y_1)=$$
$$= \Y_1^2 + \sum_{i=2}^m(\Y_i - Z_i w \Y_1)^2 + \Big(b_1- \sum_{i=2}^m Z_i w \Big)\Y_1  + \sum_{i=2}^m b_i \Y_i =$$
$$= (1 + \sum_{i=2}^m(Z_i w)^2)\Y_1^2 + \sum_{i=2}^m\Y_i^2   
- \sum_{i=2}^m Z_i w (\Y_i\Y_1 + \Y_1\Y_i)$$
\begin{equation}\label{erratoquasiovunque}
+ 
\Big(b_1- \sum_{i=2}^m Z_i w + \sum_{i=2}^m(\Y_i - Z_i w \Y_1) Z_i w \Big)\Y_1  + \sum_{i=2}^m b_i \Y_i. \end{equation}
In particular the coefficient $1 + \sum_{i=2}^m(Z_i w)^2 $ of $\Y_1^2$ 
is smooth and bounded from above and below by positive constants. 
\end{remark}
Let us now conclude the proof of the boundedness of $P$.
	
\begin{theorem}\label{albicocca}
Let $V$, $V_0$ be open sets in $\R^{n}$, with $V\subset\subset V_0$, let $g \in {\hat C}^{2, \alpha}(\partial D \cap V_0)$. Then 
there is a constant $C_1$ such that \begin{equation}\label{boundP}
\|P(g)\|_{{ C}^{2, \alpha}(D \cap V)}\leq C_1\|g\|_{{\hat C}^{2, \alpha}( \partial D \cap V_0)}.\end{equation}
\end{theorem}
\begin{proof}
Let us fix $V_1$ such that 
$V\subset\subset V_1 \subset\subset V_0$. Thanks to the previous lemma we only have to prove that 
the operator 
$$\tilde K: {C}^{1,\alpha}(D\cap V_0)\cap 
\left\{\phi:\, |\phi(0,\z)|\le C \frac{ \hat d(\z,\operatorname{supp}(g))}{|\hat B(\z, \hat d(\z,\operatorname{supp}(g)))|}, \quad \quad \quad \quad \quad 
\quad \quad\right.$$$$\quad \quad \quad \quad \quad \quad \quad
\left.\forall \hat z \; \text{s.t. }\hat d(\hat z. \supp (g)) \geq 2 \operatorname{diam}(\operatorname{supp}(g)) \right\}
\to C^{2,\alpha}(D\cap V_0)
$$
defined as 
$$\tilde K (\varphi)(x) := 
\int_{\R^{n-1}}  \Gamma_\Delta(x, (0,\hat z)) \phi(0,\hat z) d\hat z 
$$
satisfies
\begin{equation}\label{semai}
 \|\tilde K (\varphi)\|_{C^{2,\alpha}(D\cap V)}\leq\|\varphi\|_{{C}^{1,\alpha}( D\cap V_0)}.
\end{equation}
 It is standard to recognize that for every $i, j = 2, \cdots, m$, $ \Y_i \Y_j \tilde K  $ it is bounded as operator with values in $C^{\alpha}(D\cap V_0)$
(see for example \cite{GreinerStein}, \cite{NagelStein}).

 Hence we have to estimate the normal derivative. Let us begin with the derivatives 
 $ \Y_i \Y_1 \tilde K  $ with $i = 2, \cdots, m$.
 Let $\psi \in C_0^\infty(V_1)$ such that $\psi\equiv 1$ in a neighborhood of $V$ and let $x\in V$ and $\phi\in  C^{1,\alpha}(D \cap V_0)$. 
By Proposition \ref{uniformderivatives} 
there exist kernels 
$ (\Gamma_{1,i}(x, y))_{i=1, \cdots, m},$ 
$S_{1}(x, y)$ 
of type 2 such that 
\begin{align*}\partial_1 \tilde K(\varphi)(x) = 
&\int_{\R^{n-1}} 
\sum_{i=1}^m  (\Y^{z}_{i})^*  \Gamma_{1, i }(x, (0, \hat z)) \varphi(0,\z) d\z +
\int_{\R^{n-1}} S_{1}(x, (0, \hat z))  \varphi(0,\z) d\z\\
=&- \int_{\R^{n-1}}  \partial^z_1  \Gamma_{1,1 }(x, (0, \hat z))  \varphi(0,\z) d\z\\
&-\int_{\R^{n-1}} 
\sum_{i=2}^m   \Gamma_{1, i }(x, (0, \hat z)) \Y^{z}_{i} \varphi(0,\z) d\z 
+\int_{\R^{n-1}} S_{1}(x, (0, \hat z))   \varphi(0,\z) d\z.
\end{align*}
Let us estimate the first term, using the fact that 
$\partial_1^z = \partial^z_\nu$,
\begin{align*}
&\int_{\R^{n-1}} 
  \partial^z_1  \Gamma_{1, 1 }(x, (0, \hat z))    \varphi(0,\z) d\z \\
=&- \int_{\R^{n-1} \cap V_1} <\nu, \nabla \Gamma_{1, 1}(x, (0,\hat z))> \varphi(0,\hat z)  \psi(0, \hat z) d\hat z \\
& - \int_{\R^{n-1} } <\nu, \nabla \Gamma_{1,1}(x, (0,\hat z))> \varphi(0,\hat z) 
(1-\psi(0, \z)) d\hat z \\
=& 
- \int_{V_1\cap D} < \nabla \Gamma_{1,1}(x, z), \nabla(\varphi \psi)(z)  > d z  \\
&- \int_{\R^{n-1} \backslash V_1} <\nu, \nabla \Gamma_{1,1}(x, (0,\hat z))> \varphi(0,\hat z) (1-\psi(0, \z)) d \hat z.
\end{align*}
If $x\in V$ the last integral contains a $C^\infty$ kernel since $\psi=1$, on a closed set which contains $V$ in the interior. Thus, applying standard singular integral theory to all terms in the expression of $\partial_1 \tilde K$ we obtain
$$\|\partial_1 \tilde K(\varphi)\|_{C^{1, \alpha}(D \cap V)}\leq C \|\varphi\|_{{ C}^{1, \alpha}( D \cap V_0)}.$$
Analogously  for every  $i = 2, \cdots, m$ we have 
$$\|\partial_1 \Y_i\tilde K(\varphi)\|_{C^{\alpha}(D \cap V)}\leq C \|\varphi\|_{{C}^{1, \alpha}( D \cap V_0)}.$$
Finally we note that $\Delta \tilde K(x, (0, \hat y)) =0$, consequently the estimate 
of $\Y^2_1 \tilde K$ follows by difference from the estimates of all the other 
second derivatives and the expression \eqref{erratoquasiovunque}. 

Assertion \eqref{semai} is proved, so that the thesis follows. 
\end{proof}
From here it immediately follows: 
\begin{corollary}\label{operator_norms} Assume that the same assumptions as in Theorem \ref{albicocca} are satisfied. 
If $V\subset\subset V_0$, $k\in \{0,1\}$, 
$f\in C^{k,\alpha}(V_0)$, and 
$$G(f) := E_{\Gamma_{\Delta}}(f) - P((E_{\Gamma_{\Delta}}(f))_{|{\partial D \cap V_0}}),$$
there exists  a constant $C$ such that 
\begin{equation}\label{Gnabla}
\|G(f)\|_{C^{2, \alpha }(V)} \leq C \|f\|_{C^{\alpha}(V_0)}
\quad  \text{ and } \quad  \|G(\nabla f)\|_{C^{k+1, \alpha }(V)} \leq C \|f\|_{C^{k,\alpha}(V_0)}.
\end{equation}
\end{corollary}
\begin{proof}
The first inequality follows from properties of singular integrals (see \cite{libroStein}) and the boundedness of $P$ established in Theorem \ref{albicocca}. 
The last inequality follows applying Lemma \ref{gammanabla}. Indeed there exists a kernel $S$ of type 1 such that 
$$ E_{\Gamma_{\Delta}}(\nabla f) = E_{S}(f).$$
Consequently
$$G(\nabla f) = E_{S}(f) - P((E_S(f))_{|{\partial D \cap V_0}}),$$
and the assertion follows at once. 
\end{proof}

Let $D= \{(x_1,\x)\in \R \times \R^{n-1}: x_1>0\}$ be a half space as above and consider  the problem 
\begin{equation}
\label{festa2}
 \left\{\begin{array}{cl}
\Delta u=f& \text{in $D$},\\ u=g &   \text{on $\partial D$.}
\end{array}\right.\end{equation}

From Theorem \ref{mainRn}  next theorem easily follows .
\begin{theorem}\label{con f}
If $f\in C^\infty _0(V_0)$ and $g\in C_0^\infty(\partial D \cap V_0)$ and 
$$G(f) = E_{\Gamma_{\Delta}}(f) - P(E_{\Gamma_{\Delta}}(f))_{|\partial D \cap V_0}),$$
then the function  $u=G(f) + P(g)$
solves the problem $$\Delta u=f \ \text{in }D, \quad u=g   \ \text{on }\partial D\cap V_0.$$
\end{theorem}
As a consequence of the previous theorem, we immediately get an approximate 
representation formula for a smooth function $u$.
\begin{lemma}
Let $V\subset\subset V_0$ and let $u\in C^{\infty}_0(V)$. Let us call $\Delta u=f$ and $g=u_{|\partial D \cap V_0}$, 
and let $\phi \in C^\infty_0(V_0)$, $\phi =1$ on $V$. 
Then 
\begin{equation} \label{ravioli}u = \phi v + E_{\Gamma_\Delta}\Big(f (1-\phi) + v \sum _{i=1}^m b_i X_i\phi\Big) -
E_{S} (v \nabla\phi),\end{equation}
where $v= G(f)+P(g)$ and $b_i$ are the coefficients of the operator $\Delta$ in \eqref{laplacoperator}. 
\end{lemma}
\begin{proof}
Setting $v= G(\Delta  u)+P(u|_{\partial D \cap V_0})$
we have by Theorem \ref{con f}
$$
\left\{\begin{array}{lll}
\Delta (u - \phi v)&=f (1-\phi) + \nabla v \nabla \phi + v \Delta \phi &\quad \text{in}\, V_0\cap  D,
\\
u - \phi v &= 0 &\text{on}\, \partial (V_0\cap D)\,.
\end{array}\right.
$$
where we have extended $u - \phi v$ on the whole space with $0$. 
We deduce by \eqref{laplacoperator}
$$
u = \phi v + E_{\Gamma_\Delta}\Big(f (1-\phi) + \nabla v \nabla \phi + v \Delta \phi\Big)=
$$$$=
 \phi v +E_{\Gamma_\Delta}\Big(f (1-\phi) + v \sum_i b_i X_i\phi \Big) - E_{\Gamma_\Delta}(\nabla  (v \nabla\phi)).$$
Now applying Lemma \ref{gammanabla} we obtain 
$$u = \phi v + E_{\Gamma_\Delta}\Big(f (1-\phi) + v \sum _ib_i X_i\,\phi\Big) -
E_{S} (v \nabla\phi).$$
\end{proof}

\subsection{Schauder estimates} 
We can now complete the proof of the Schauder estimates, stated in the introduction: 
\begin{proof}[Proof of Theorem \ref{c:schauderGroups}]
Let $u$ be a solution of $\Delta u=f$ and $u_{|\partial D}=g$. 
We will prove the a priori estimates for $u$ under the assumption that 
$f\in C^\infty(\bar D)$, $g\in C^\infty(\partial D)$ and we will obtain 
the thesis for $f\in C^{\alpha}(\bar D)$, $g\in \hat C^{2, \alpha}(\partial D)$  by a density argument. 
For smooth data, by \cite{KN} there exists a unique solution   $u\in C^\infty(D)$, smooth up to the boundary at non characteristic points. 

We first note that 
$$\|u\|_\infty \leq \|g\|_\infty$$
via the maximum principle. In addition,  
extending $g$ in the interior of $D$ to a function of class $C^{2, \alpha}$ such that 
$\|g\|_{C^{2, \alpha}(D)}\leq \|g\|_{\hat C^{2, \alpha}(\partial D)},$
we see that $u-g$ is a solution of $\Delta(u-g)= f- \Delta g$ in $D$ and   $ u-g=0 $ on ${\partial D}$, hence the Moser iteration technique (see  \cite{moser})  ensures that there exists a value of $\beta$ such that 
 $u-g\in C^{\beta}(\bar D),$ and 
\begin{equation}\label{normabeta}
\|u\|_{C^\beta(\bar D)} \leq C\big(\|f\|_{C^\alpha(\bar D)} + \|g\|_{\hat C^{2, \alpha}(\partial D)}\big).
\end{equation} 
We can choose a non characteristic point, say $0\in \partial D$, 
and denote $V_0$ a neighborhood of $0$ such that 
the subriemannian normal $$\nu(s) \ne 0 \ \text{for every } s\in \partial D\cap V_0.$$
Then we can perform the change of variable described in Section \ref{nuova}
on a set $V\subset\subset V_0$.  
Through this change of variables
the vector fields $X_i $ can be represented as in \eqref{campinonomo}
$$d\Xi(X_1) = \partial_{x_1}, \quad 
d\Xi(X_i) = \partial_{x_i} + \sum_{deg(j) > deg(i)} a_{i,j}(x_1 + w(\hat x), \hat x)\partial_{x_j} + X_i w (\x)\partial_{y_1},$$
so that the results of the previous section apply. 
Let $\phi\in C_0^{\infty}(V)$, let $V_1$ be an open set such that $V\subset\subset V_1$,  $\phi_1\in C_0^{\infty}(V_1)$ and identically $1$ on $V$. Define 
\begin{equation}\label{preraviolo}
v:=G(\Delta (\phi u)) + P((\phi u)_{|\partial D \cap V})= G\Big(
f \phi  + \nabla \phi \nabla u+ \Delta \phi u\Big) + P((\phi u)_{|\partial D \cap V}).
\end{equation}
By \eqref{ravioli}  we get 
\begin{equation}\label{postraviolo}
	\phi u = 
\phi_1 v + E_{\Gamma_\Delta}\Big(f (1-\phi_1) - v \sum_i b_i X_i\phi_1\Big) + 
E_S (v \nabla\phi_1). 
\end{equation}
Then, from previous expressions and using \eqref{Gnabla}, for nested open sets $V \subset\subset V_3\subset \subset V_2\subset \subset V_1 \subset\subset V_0$  and for every  $\gamma\leq \alpha$ we get that 
\begin{equation}\label{stimabeta}
\|\phi u\|_{C^{1,\gamma}(V_3\cap D)} \leq C\big(\|v\|_{C^{1,\gamma}(V_2\cap D)} + \|f\|_{C^{\alpha}(V_2\cap D)}\big)\end{equation}
$$\leq C\big(\|u\|_{C^{\gamma}(V_1\cap D)} + \|f\|_{C^{\alpha}(\bar D)} + \|g\|_{\hat C^{2,\alpha}(\partial D)} \big).$$
In particular, using this inequality and the uniform estimate of $\|\phi u\|_{C^{\beta}(\bar D)}$ provided by \eqref{normabeta} we get for  $V \subset\subset V_4 \subset\subset V_3 $
$$\|\phi u\|_{C^{\alpha}(V_4\cap D)} \leq C\|\phi u \|_{C^{1,\beta}(V_3\cap D)} 
\leq C\big(\|f\|_{C^{\alpha}(\bar D)} + \|g\|_{\hat C^{2,\alpha}(\partial D)} \big). $$
Having an estimate of $\|\phi u\|_{C^{\alpha}}$ we apply again  \eqref{normabeta} with $\gamma =\alpha$ 
and we have for $V \subset\subset V_5\subset \subset V_4 $
$$\|\phi u \|_{C^{1,\alpha}(V_5\cap D)} 
\leq C\big(\|f\|_{C^{\alpha}(\bar D)} + \|g\|_{\hat C^{2,\alpha}(\partial D)} \big). $$
Finally, we iterate the same argument applying again \eqref{Gnabla} to \eqref{preraviolo} and \eqref{postraviolo}. Therefore we get 
$$
\|\phi u\|_{C^{2, \alpha} (V\cap D)}\leq C\big(\|\phi u\|_{C^{1,\alpha}(V_5\cap D)}  + \|f\|_{C^{\alpha}(\bar D)} + \|g\|_{\hat C^{2, \alpha}(\partial D)} \big)\leq $$$$\leq
C\big(\|f\|_{C^{\alpha}(\bar D)} + \|g\|_{\hat C^{2, \alpha}(\partial D)} \big).
$$
This concludes the proof.
\end{proof}

\end{document}